\documentclass[reqno, 11pt]{amsart}
\usepackage[text={155mm,235mm},centering]{geometry}
\input diagxy
\input xy

\usepackage[active]{srcltx} % SRC Specials

\newtheorem{thm}{Theorem}[section]

\newtheorem{lem}[thm]{Lemma}
\newtheorem{prop}[thm]{Proposition}

\theoremstyle{definition}
\newtheorem{defn}[thm]{Definition}
\theoremstyle{property}

\theoremstyle{remark}

\newtheorem{ex}[thm]{Example}
\numberwithin{equation}{section}
\newtheorem*{thm*}{Theorem}
\usepackage[mathscr]{eucal}
\usepackage[all]{xy}
\usepackage{mathrsfs}
\usepackage{xypic}
\usepackage{amsfonts}
\usepackage{amsmath}
\usepackage{amsthm,bm}
\usepackage{amssymb,tikz-cd}
\usepackage{latexsym}
\usepackage{tabularx}
\usepackage{graphicx}
\usepackage{pict2e}
\usepackage{hyperref}
\usepackage{tikz}
\usepackage{textcomp}
\usepackage[scr=rsfs]{mathalfa}
\usepackage[all]{xy}
\usepackage{tikz-cd}
\definecolor{ceruleanblue}{rgb}{0.16, 0.32, 0.75}
\hypersetup{colorlinks=true,allcolors=ceruleanblue}
\allowdisplaybreaks
\begin{document}

\title[A Transverse Averaging Operator and Cohomology of G/H]
{A Transverse Averaging Operator and the Cohomology of Homogeneous Quotients by Non-closed Subgroups}
%---------------------------------------
\author{Yi Lin}

\address{Y. Lin \\
Department of Mathematical Sciences \\
Georgia Southern University \\
Statesboro, GA, 30460 USA}
\email{yilin@georgiasouthern.edu}
%================================

%--------------------------------------

\date{\today}

%\dedicatory{ }

%\commby{}
%---------------------------------
\begin{abstract}

In this article, we introduce a transverse averaging operator for basic forms
on a complete Riemannian foliation with compact leaf closure space, equipped
with an isometric transverse Lie algebra action. In contrast to the classical
averaging operator in equivariant geometry, which is defined by integration
over a compact Lie group, our operator is constructed purely from
infinitesimal transverse data and does not require any global group action. We
prove that every closed basic form is sent to an invariant basic form
representing the same basic cohomology class.

The main application is formulated independently of foliation theory. We
compute the diffeological de Rham cohomology of the homogeneous quotient
$G/H$, where $G$ is a connected Lie group, not necessarily compact, and $H$
is a connected Lie subgroup, not necessarily closed. Let $\mathfrak g$ and
$\mathfrak h$ be the Lie algebras of $G$ and $H$, respectively. Assuming that
$\mathfrak g$ is of compact type and that $G/\overline{H}$ is compact, we
prove that
$
H^\bullet_{\mathrm{dR}}(G/H)
\cong
H^\bullet(\mathfrak g,\mathfrak h)$.
When $\mathfrak h$ is an ideal in $\mathfrak g$, the compact-type assumption
on $\mathfrak g$ can be dropped, and under the sole assumption that
$G/\overline{H}$ is compact we obtain
$
H^\bullet_{\mathrm{dR}}(G/H)
\cong
H^\bullet(\mathfrak g/\mathfrak h)$. 
These results extend the classical Chevalley--Eilenberg computation from
compact Lie groups and closed subgroups to homogeneous quotients $G/H$ by possibly
non-closed subgroups $H$.

\end{abstract}

%-----------------------------------------------------------
\maketitle
%-----------------------------------------------------------
%\tableofcontents

\section{Introduction}
A classical theorem of Chevalley and Eilenberg~\cite{CE48} identifies the de
Rham cohomology of $G/H$ with the relative Lie algebra cohomology
$H^\bullet(\mathfrak g,\mathfrak h)$ when $G$ is compact and connected and
$H\subset G$ is closed and connected. When $G$ is noncompact or $H$ is
nonclosed, the hypotheses of this classical theorem are no longer satisfied;
moreover, if $H$ is nonclosed, the quotient $G/H$ need not be a smooth
manifold. Nevertheless, $G/H$ carries a natural quotient diffeology in the
sense of \cite{I13}, and one may therefore consider its diffeological de Rham
cohomology. It is natural to ask under what conditions the classical
Lie-algebraic computation remains valid for such quotients.

We show in this paper that the geometry of foliations provides a new and
effective approach to this question. Indeed, for a connected Lie subgroup
$H$, the orbits of the right $H$-action on $G$ define a regular foliation
$\mathcal F$ whose leaf space is precisely the diffeological quotient $G/H$.
By a theorem of Hector, Mac\'{\i}as-Virg\'os, and
Sanmart\'{\i}n-Carb\'on~\cite[Theorem~3.5]{HMS}, the diffeological de Rham
cohomology of $G/H$ is naturally identified with the basic cohomology of
$(G,\mathcal F)$. The problem is therefore reduced to a computation in basic
cohomology, which was introduced by Reinhart~\cite{Re59} as a cohomology
theory for the leaf space of a foliation. As we show later, the right
$H$-orbit foliations arising in our main applications are Riemannian, so this
reduction places the problem within the well-developed theory of basic
cohomology for Riemannian foliations.

Within the theory of Riemannian foliations, Goertsches and
T\"oben~\cite{GT10} introduced transverse Lie algebra actions and the
associated equivariant basic cohomology. The canonical transverse action of
the structural Lie algebra of a Killing foliation preserves the given
transverse Riemannian metric and is therefore isometric. Goertsches and
T\"oben proved, among other results, a Borel-type localization theorem in this
special but important setting. Lin and Sjamaar~\cite{LS20} subsequently
developed the theory for general isometric transverse Lie algebra actions on
Riemannian foliations and extended the localization theorem to this broader
setting. More recently, Lin and Wang~\cite{LW24} developed an equivariant
Morse--Bott theory for isometric transverse Lie algebra actions on general
Riemannian foliations.

As observed by Goertsches and T\"oben~\cite[Sec.~3.6]{GT10}, 
no general averaging procedure is available in this framework.
In classical equivariant
geometry, averaging over a compact and connected Lie group allows one to
replace arbitrary closed differential forms by invariant representatives
without changing their cohomology classes. For an isometric transverse Lie
algebra action, however, there need not be a global compact group action on
the foliated manifold $M$, so the classical averaging construction is
unavailable. In this paper, we construct a transverse averaging operator from
infinitesimal transverse data, without assuming the existence of any global
group action on $M$.

Our construction combines several ingredients from the theory of Riemannian
foliations, notably Molino's structural theory, the framework for isometric
transverse Lie algebra actions developed by Lin and Sjamaar, Sergiescu's
transverse integration, and basic Poincar\'e duality. Following \cite{LS20},
we pass to the transverse orthonormal frame bundle $\pi_M:P\to M$. The
isometric transverse $\mathfrak g$-action lifts to $P$ and projects to an
isometric Lie algebra action on the Molino manifold $W$. Under our standing
assumption that the leaf closure space $M/\overline{\mathcal F}$ is compact,
the Molino manifold $W$ is compact as well. The projected infinitesimal
action on $W$ therefore consists of complete Killing vector fields and
integrates to an isometric action of the simply connected Lie group
$\widetilde G$ with Lie algebra $\mathfrak g$. Since the isometry group of a
compact Riemannian manifold is compact, the closure $N$ of the image of
$\widetilde G$ in $\mathrm{Iso}(W)$ is a compact and connected Lie group.
Although $N$ does not act directly on $M$, it induces an action on the space
of basic forms on $M$ through the geometry of the transverse orthonormal
frame bundle $P$. The averaging operator is defined by integrating this
induced action over $N$, and the fact that it preserves basic cohomology
classes is established using Sergiescu's transverse integration~\cite{Ser85}
and basic Poincar\'e duality.

The main results of this paper are as follows.

\begin{thm}\label{average-operator}
Consider an isometric transverse action of a Lie algebra $\mathfrak{g}$ on a
complete Riemannian foliation whose leaf closure space is compact. Then there exists
an averaging operator
\[
\mathcal{A}: \Omega^\bullet(M, \mathcal{F}) \to \Omega^\bullet(M, \mathcal{F})^{\mathfrak{g}},
\]
such that for every closed basic form $\beta \in \Omega^k(M,\mathcal{F})$,
the form $\mathcal{A}(\beta)$ is closed and $\mathfrak{g}$-invariant, and
satisfies
\[
[\mathcal{A}(\beta)] = [\beta] \in H^k(M,\mathcal{F}).
\]
\end{thm}

The transverse averaging operator has several applications in foliation
theory. As a first application, we describe the $E_1$-term of the spectral
sequence associated to the Cartan complex computing the equivariant basic
cohomology $H_{\mathfrak g}(M,\mathcal F)$. In the special case of Killing
foliations, for the canonical transverse action of the structural Lie
algebra, Goertsches and T\"oben~\cite[Thm.~3.23]{GT10} identified this
$E_1$-term by using the fact that every basic form is automatically invariant
under the action. For a general transverse Lie algebra action, however, they
pointed out that the situation is less clear, since
$\Omega(M,\mathcal F)^{\mathfrak g}$ may be a proper subspace of
$\Omega(M,\mathcal F)$ and no averaging process was available
\cite[Sec.~3.6]{GT10}. The transverse averaging operator constructed in
Theorem~\ref{average-operator} resolves this difficulty for isometric
transverse Lie algebra actions on complete Riemannian foliations with compact
leaf closure space.

\begin{thm}\label{thm:E1}
Let $(M,\mathcal F)$ be a complete Riemannian foliation with compact leaf
closure space, equipped with an isometric transverse $\mathfrak g$-action.
Then the $E_1$-term of the spectral sequence associated to the Cartan complex
of equivariant basic cohomology is naturally isomorphic to
\[
E_1 \cong (S\mathfrak{g}^*)^{\mathfrak{g}} \otimes H^\bullet(M,\mathcal{F}).
\]
\end{thm}

A second application concerns Lie foliations. We say that a Lie
$\mathfrak g$-foliation is of \emph{compact type} if its structural Lie
algebra $\mathfrak g$ is of compact type. Such a foliation carries a
canonical isometric transverse $\mathfrak g$-action. By adapting the
transverse averaging construction to its natural transverse parallelism, we
obtain the following cohomological description.

\begin{thm}\label{thm:liefoliation}
Let $(M,\mathcal{F})$ be a complete Lie $\mathfrak{g}$-foliation of compact type with compact leaf closure space. Assume that $M$ is connected. Then
\[
H^\bullet(M,\mathcal{F}) \cong H^\bullet(\mathfrak{g}).
\]
\end{thm}

The principal application, however, returns to the question posed at the
beginning of the introduction---a question formulated entirely outside
foliation theory. By applying the transverse averaging operator to the
foliation of $G$ by right $H$-orbits, we obtain the following extensions of
the classical Chevalley--Eilenberg theorem. Although foliation geometry
provides the method of proof, the resulting statements concern only $G$, $H$,
$\mathfrak g$, and $\mathfrak h$.

\begin{thm}\label{thm:GH1}
Let $G$ be a connected Lie group, not necessarily compact, and let $H$ be a connected Lie subgroup of $G$, not necessarily closed. Let $\mathfrak g$ and $\mathfrak h$ be the Lie algebras of $G$ and $H$, respectively. Suppose that $G/\overline{H}$ is compact and that $\mathfrak g$ is of compact type. Then the diffeological de Rham cohomology $H^\bullet_{\mathrm{dR}}(G/H)$ is isomorphic to the relative Lie algebra cohomology $H^\bullet(\mathfrak g,\mathfrak h)$.
\end{thm}

When $\mathfrak h$ is an ideal of $\mathfrak g$, the compact-type assumption
on $\mathfrak g$ can be removed.

\begin{thm}\label{thm:GH}
Let $G$ be a connected Lie group, not necessarily compact, and let $H$ be a connected Lie subgroup of $G$, not necessarily closed. Let $\mathfrak g$ and $\mathfrak h$ be the Lie algebras of $G$ and $H$, respectively. Suppose that $\mathfrak h$ is an ideal of $\mathfrak g$ and that $G/\overline{H}$ is compact. Then the diffeological de Rham cohomology $H^\bullet_{\mathrm{dR}}(G/H)$ is isomorphic to the Lie algebra cohomology $H^\bullet(\mathfrak g/\mathfrak h)$.
\end{thm}

In the special case where $H$ is connected and dense in $G$,
Theorem~\ref{thm:GH} recovers a recent result of Clark and
Ziegler~\cite{CZ26}, who proved the same conclusion for an arbitrary dense Lie
subgroup $H$ of $G$. Indeed, if $H$ is dense in $G$, then
$\mathfrak g/\mathfrak h$ is automatically abelian, hence of compact type,
and $G/\overline{H}$ is a point.  Clark and Ziegler also identified a common 
generalization of their
dense-subgroup theorem and the classical Chevalley--Eilenberg theorem to
arbitrary subgroups as a natural next problem. They explained, however, that
the two known arguments reduce the computation to $G$-invariant forms by
fundamentally different mechanisms: in the dense-subgroup case, invariance
holds already at the cochain level by density, whereas in the classical
homogeneous-space case it is obtained only in cohomology by averaging, a
method that essentially requires $G$ to be compact.

Theorems~\ref{thm:GH1} and~\ref{thm:GH} bridge these two settings under the
additional assumptions that $H$ is connected and that $G/\overline H$ is
compact. The transverse averaging operator provides the missing mechanism: it
produces invariant representatives without requiring either the density of
$H$ or the compactness of $G$. Moreover, the compactness assumption on
$G/\overline H$ singles out a genuinely rigid range of validity: outside this
regime, we give simple counterexamples showing that the conclusion fails. On
the other hand, our results do not fully subsume the Clark--Ziegler theorem,
since their theorem does not require $H$ to be connected.

The remaining distinction between our result and the full Clark--Ziegler
theorem is the possible disconnectedness of $H$. Treating
that case requires one to keep track of the residual action of the component
group and consequently involves more substantial diffeological input. We do
not pursue that refinement here. In the present paper, diffeology enters only
at a black-box level, through the theorem of Hector,
Mac\'{\i}as-Virg\'os, and Sanmart\'{\i}n-Carb\'on identifying the
diffeological de Rham cohomology of a leaf space with the basic cohomology of
the corresponding foliation. This is the only diffeological input needed for
Theorems~\ref{thm:GH1} and~\ref{thm:GH}. A foliation-theoretic generalization
encompassing the full Clark--Ziegler theorem is developed separately in
\cite{L26}, where diffeological methods are used more substantially.

The paper is organized as follows. Section~\ref{sec:background} recalls the
necessary background on transverse geometric structures and equivariant basic
cohomology. Section~\ref{integration} reviews Sergiescu's approach to
transverse integration via Molino theory, together with the basic
Poincar\'e duality theorem. Section~\ref{sec:averaging-operator} is devoted to
the construction of the transverse averaging operator, the proof of
Theorem~\ref{average-operator}, and the computation of the $E_1$-term in
Theorem~\ref{thm:E1}. Section~\ref{App:Lie-foliation} applies the transverse
averaging operator to compute the basic cohomology of Lie foliations of
compact type. Finally, Section~\ref{De Rham of G/H} applies the transverse
averaging method to prove Theorems~\ref{thm:GH1} and~\ref{thm:GH}, computing
the diffeological de Rham cohomology of the quotient $G/H$, where $H$ is a
Lie subgroup of $G$, not necessarily closed.

\section{Transverse geometric structures on foliations}\label{sec:background}

This section reviews the transverse geometric and algebraic background used
throughout the paper. After fixing notation for transverse vector fields and
basic forms, we recall transverse Riemannian metrics, transverse Killing vector
fields, and the Cartan model for equivariant basic cohomology. We also review
relative Chevalley--Eilenberg cohomology, which will be used in the proofs of
Theorems~\ref{thm:GH1} and~\ref{thm:GH}.

 Let $\mathcal{F}$ be a foliation on a smooth manifold $M$, and let $T\mathcal{F}$ be the tangent bundle of the foliation.
Throughout this paper we denote by
$\mathfrak{X}(\mathcal{F})\subset\mathfrak{X}(M)$ the subspace of vector fields tangent to the leaves of $\mathcal{F}$.
We say that a vector field $X\in\mathfrak{X}(M)$ is \emph{foliate}, if $[X,Y]\in\mathfrak{X}(\mathcal{F})$ for all $Y\in \mathfrak{X}(\mathcal{F})$.
We will denote by
$\mathfrak{R}(\mathcal{F})$ the space of foliate vector fields on $(M,\mathcal{F})$.
Clearly we have that
$\mathfrak{X}(\mathcal{F}) \subset \mathfrak{R}(\mathcal{F})$.
In this context,  a \emph{transverse vector field} is an equivalence class in the quotient space $\mathfrak{R}(\mathcal{F})/\mathfrak{X}(\mathcal{F})$. The space of transverse vector fields, denoted by $\mathfrak{X}(M, \mathcal{F})$, forms a Lie algebra with a Lie bracket inherited from the natural one on $\mathfrak{R}(\mathcal{F})$.
The space of \emph{basic forms} on $M$ is
\[
\Omega^\bullet(M,\mathcal F)
:=
\{\alpha\in \Omega^\bullet(M)\mid
\iota_X\alpha=\mathcal L_X\alpha=0,\ \forall X\in \mathfrak X(\mathcal F)\}.
\]
Since $d$ preserves basic forms, $(\Omega^\bullet(M,\mathcal F),d)$ is a subcomplex
of the de Rham complex. Its cohomology $H^\bullet(M,\mathcal F)$ is called the
\emph{basic cohomology}.

%Let $N=TM/T\mathcal{F}$ be the normal bundle of the foliation.
%A moment's consideration shows that for any foliate vector field $X$, and for any $(r, s)$-type tensor \[\sigma\in C^{\infty}(\underbrace{Q^*\otimes\cdots\otimes Q^*}_{r}\otimes\underbrace{Q\otimes\cdots \otimes Q}_s),\] the Lie derivative
%$\mathcal{L}(X)\sigma$ is well defined.

\begin{defn}
A \emph{transverse Riemannian metric} on a foliation $(M,\mathcal{F})$ is a Riemannian metric $g$ on the normal bundle $TM/T\mathcal{F}$ of the foliation, such that $\mathcal{L}_Xg=0$,
for $X\in\mathfrak{X}(\mathcal{F})$.
We say that $(\mathcal{F}, g)$ is a Riemannian foliation if there exists a transverse Riemannian metric $g$ on $(M,\mathcal{F})$.
\end{defn}

\begin{defn}
A Riemannian metric $g_{TM}$ on a foliated manifold $(M,\mathcal F)$ is
\emph{bundle-like} if for every open set $U\subset M$ and every pair of
foliate vector fields $v,w$ on $U$ that are orthogonal to the leaves,
the function $g_{TM}(v,w)$ is basic on $U$.
\end{defn}

A bundle-like Riemannian metric $g_{TM}$ on $(M,\mathcal F)$ induces a transverse
Riemannian metric $g$ by identifying $N\mathcal F$ with the
$g_{TM}$-orthogonal complement of $T\mathcal F$ and restricting $g_{TM}$ to
$N\mathcal F$. Conversely, every transverse Riemannian metric arises in this way
from a bundle-like Riemannian metric. We say that a Riemannian foliation
$(M,\mathcal F,g)$ is \emph{complete} if there exists a complete bundle-like
Riemannian metric on $M$ inducing $g$.  

%Let $q$ be the codimension of a Riemannian foliation $\mathcal{F}$ on a compact manifold $M$.
%If $H^{q}(M,\mathcal{F})=\mathbb{R}$, we say that $\mathcal{F}$ is \emph{homologically orientable}.

Let $(M,\mathcal{F})$ be a Riemannian foliation with transverse metric $g$, and let $\overline{X}\in \mathfrak{X}(M,\mathcal{F})$ be a transverse vector field represented by a foliate vector field $X$. Define
$\mathcal{L}_{\overline{X}}g:=\mathcal{L}_Xg$. This is well defined, i.e.\ independent of the choice of representative $X$. We say that $\overline{X}$ is \emph{transversely Killing} if $\mathcal{L}_{\overline{X}}g=0$. If $\overline{X}$ and $\overline{Y}$ are transversely Killing, then the Cartan identities imply that $[\overline{X},\overline{Y}]$ is again transversely Killing. Thus the space of transversely Killing vector fields, denoted by $\mathrm{iso}(M,\mathcal{F})$, is a Lie subalgebra of $\mathfrak{X}(M,\mathcal{F})$, called the \emph{transverse Killing Lie algebra}.

\subsection{Transverse Lie algebra action and the equivariant basic cohomology}

\begin{defn}\label{transverse-action}
A \emph{transverse action} of a Lie algebra $\mathfrak{g}$ on a foliated manifold $(M,\mathcal{F})$ is a Lie algebra homomorphism \begin{equation}\label{equ1.1}
\mathfrak{g}\rightarrow\mathfrak{X}(M, \mathcal{F}).  \end{equation}
When the foliation $\mathcal{F}$ under consideration is Riemannian with a given transverse Riemannian metric, a transverse action of $\mathfrak{g}$ on $(M,\mathcal{F})$ is said to be isometric, if the image of the map (\ref{equ1.1}) lies inside the transverse Killing Lie algebra $\mathrm{iso}(M, \mathcal{F})$.
\end{defn}

Suppose that a Lie algebra $\mathfrak g$ acts transversely on a foliated manifold
$(M,\mathcal F)$. For each $\xi\in\mathfrak g$, let $\overline{\xi}_M$ denote
the corresponding transverse vector field, and choose a foliate vector field
$\xi_M$ representing it. For $\alpha\in\Omega(M,\mathcal F)$, define
\[
\iota(\xi)\alpha:=\iota_{\xi_M}\alpha,\qquad
\mathcal L(\xi)\alpha:=\mathcal L_{\xi_M}\alpha.
\]
Since $\alpha$ is basic, these operators are independent of the choice of
representative $\xi_M$. By \cite[Proposition~3.12]{GT10}, a transverse
$\mathfrak g$-action endows the basic de Rham complex $\Omega(M,\mathcal F)$
with the structure of a $\mathfrak g^\star$-algebra in the sense of
\cite[Chapter~2]{GS99}. Its Cartan model is therefore
\[
\Omega^\bullet_{\mathfrak g}(M,\mathcal F):=
\bigl[S\mathfrak g^*\otimes\Omega(M,\mathcal F)\bigr]^{\mathfrak g},
\]
whose elements may be identified with equivariant polynomial maps
$\mathfrak g\to \Omega(M,\mathcal F)$; these are called
\emph{equivariant basic differential forms}.
The Cartan complex is bigraded by
\[
\Omega_{\mathfrak g}^{i,j}(M,\mathcal F):=
\bigl[S^i\mathfrak g^*\otimes\Omega^{j-i}(M,\mathcal F)\bigr]^{\mathfrak g},
\]
and carries two differentials: the vertical differential
$ d:=1\otimes d$, 
and the horizontal differential $d'$ defined by
\[
(d'\alpha)(\xi):=-\,\iota(\xi)\alpha(\xi),
\qquad \xi\in\mathfrak g.
\]
Thus, as a single complex,
\[
\Omega_{\mathfrak g}^k(M,\mathcal F)
=
\bigoplus_{i+j=k}\Omega_{\mathfrak g}^{i,j}(M,\mathcal F),
\]
with total differential $d_{\mathfrak g}=d+d'$.

The equivariant basic de Rham cohomology
$H^\bullet_{\mathfrak g}(M,\mathcal F)$ is defined to be the cohomology of the
Cartan complex
\[
\bigl(\Omega_{\mathfrak g}(M,\mathcal F),d_{\mathfrak g}\bigr).
\]
The bigrading above yields the spectral sequence whose $E_1$-term is computed in
Theorem~\ref{thm:E1}.\subsection{Relative Chevalley--Eilenberg cohomology}

Let $\mathfrak g$ be a Lie algebra and let $\mathfrak h\subset \mathfrak g$ be a
Lie subalgebra. Recall that the \emph{Chevalley--Eilenberg complex} of $\mathfrak g$ is
\[
(C_{CE}^\bullet(\mathfrak g),d_{CE}) := (\Lambda^\bullet \mathfrak g^*, d_{CE}),
\] where the differential
$d_{CE}: \Lambda^k\mathfrak{g}^* \to \Lambda^{k+1}\mathfrak{g}^*$
is defined by
\[
(d_{CE}\eta)(\xi_0, \xi_1, \ldots, \xi_k)
= \sum_{0 \leq i < j \leq k} (-1)^{i+j}
\eta([\xi_i, \xi_j], \xi_0, \ldots, \hat{\xi}_i,
\ldots, \hat{\xi}_j, \ldots, \xi_k)
\]
for $\eta \in \Lambda^k\mathfrak{g}^*$ and
$\xi_0, \ldots, \xi_k \in \mathfrak{g}$, where $\hat{\xi}_i$
denotes omission of $\xi_i$. One verifies that $d_{CE}^2 = 0$
as a consequence of the Jacobi identity. The cohomology of this
complex is called the \emph{Lie algebra cohomology} of
$\mathfrak{g}$, denoted $H^\bullet(\mathfrak{g})$.

For $X\in \mathfrak g$, let $\iota_X$ denote contraction by $X$, and let
$\operatorname{ad}^*(X)$ denote the coadjoint action of $X$ on
$\mathfrak g^*$, extended to $\wedge^\bullet \mathfrak g^*$ as a degree-zero
derivation. Thus, for $\alpha\in \wedge^k\mathfrak g^*$,
\[
(\operatorname{ad}^*(X)\alpha)(Y_1,\dots,Y_k)
=
-\sum_{i=1}^k
\alpha(Y_1,\dots,[X,Y_i],\dots,Y_k),
\]
and
\[
\operatorname{ad}^*(X)=d_{CE}\iota_X+\iota_X d_{CE}.
\]

\begin{defn}
The relative Chevalley--Eilenberg complex of the pair $(\mathfrak g,\mathfrak h)$
is the subcomplex
\[
C_{CE}^\bullet(\mathfrak g,\mathfrak h)
:=
\{\alpha\in \wedge^\bullet \mathfrak g^* \mid
\iota_X\alpha=0,\ \operatorname{ad}^*(X)\alpha=0,\ \forall X\in \mathfrak h\}.
\]
$C_{CE}^\bullet(\mathfrak g,\mathfrak h)$ is preserved by $d_{CE}$.
The cohomology of the cochain complex $(C_{CE}^\bullet(\mathfrak g,\mathfrak h),d_{CE})$
is called the relative Lie algebra cohomology of $(\mathfrak g,\mathfrak h)$ and is
denoted by $H^\bullet(\mathfrak g,\mathfrak h)$.
\end{defn}

Clearly, if $\mathfrak h=0$, then
$C_{CE}^\bullet(\mathfrak g,\mathfrak h)=C_{CE}^\bullet(\mathfrak g)$ and
$H^\bullet(\mathfrak g,\mathfrak h)=H^\bullet(\mathfrak g)$.
We also record the following standard fact, which will be used in the paper; see, for example,
\cite[Chapter III]{K88}.

\begin{lem}\label{CE-complex}
If $\mathfrak h$ is an ideal of $\mathfrak g$, then the quotient map
$q:\mathfrak g\to \mathfrak g/\mathfrak h$ induces a natural isomorphism of
cochain complexes
\[
q^*:\wedge^\bullet(\mathfrak g/\mathfrak h)^*
\xrightarrow{\;\cong\;}
C_{CE}^\bullet(\mathfrak g,\mathfrak h).
\]
Consequently,
\[
H^\bullet(\mathfrak g,\mathfrak h)\cong H^\bullet(\mathfrak g/\mathfrak h).
\]
\end{lem}

\section{Sergiescu's transverse integration and basic Poincar\'e duality} \label{integration}

We begin with a brief review of the foundational aspects of Molino's structural
theory of Riemannian foliations \cite{Mo88} that will be used in this paper.
Throughout we assume that $(M,\mathcal F)$ is a complete Riemannian foliation of
codimension $q$ equipped with a transverse Riemannian metric $g$. Let
$\pi_M:P\to M$ denote the corresponding transverse orthonormal frame bundle on
$M$. When $\mathcal F$ is transversely oriented, one may replace $P$ with the
\emph{oriented transverse orthonormal frame bundle} $P^+$, consisting of
positively oriented frames, which ensures that $P^+$ and the Molino manifold
$W$ are connected when $M$ is connected. Whenever $P^+$ is defined, all
constructions and arguments in this paper stated for $P$ apply equally to
$P^+$, with only obvious notational changes. We adopt this convention in the
proof of Theorem~\ref{thm:liefoliation} in Section~\ref{App:Lie-foliation},
where we use the connectedness of the Molino manifold $W$; the constructions in
the present section and in Section~\ref{sec:averaging-operator} do not depend on
this choice.

%Choose a basis $\zeta_1, \cdots, \zeta_{r}$ of $o(q)$. 

%Let $X_{q+i}$ be the fundamental vector field on $P$ generated by $\zeta_i$, $1\leq i\leq r$. Then $ X_{q+1}, \cdots, X_{q+r}$ must be foliate and linearly independent.
 The foliation $\mathcal{F}$ naturally lifts to a foliation $\mathcal{F}_P$ on $P$ of codimension $q+r$, which is invariant under the action of the structure group $O(q)$ on $P$. Here $r:=\frac{q(q-1)}{2}$ is the dimension of the structure group $O(q)$. A key feature of the lifted foliation $ \mathcal{F}_P$ is that it is transversely parallelizable. Consequently, there is a fibration $\pi_W: P\rightarrow W$ whose fibers are precisely the leaf closures of $\mathcal{F}_P$. The base manifold $W$ is called the \emph{Molino manifold} of the Riemannian foliation $(M, \mathcal{F})$. 
Together, these maps organize into the Molino diagram (\ref{Molino-diagram}).
\begin{equation} \label{Molino-diagram}
\begin{tikzcd}[column sep=large, row sep=large]
& P \arrow[dl, "\pi_M"'] \arrow[dr, "\pi_W"] & \\
M & & W
\end{tikzcd}
\end{equation}

By definition, a point $z\in P$ is an orthonormal frame of $N_{\pi(z)}\mathcal{F}$, or equivalently, an isometry $z: \mathbb{R}^q\rightarrow N_{\pi(z)}\mathcal{F}$. The \emph{fundamental one form}
$\eta$ on $P$ is given by
\begin{equation}\label{f-one-form} \eta (X_z)= z^{-1}( \pi_{*z}(X_z)), \, \, \, X_z\in T_zP.\end{equation} 
Here by a standard abuse of notation, $\pi_{*z}$ denotes the composition of the quotient map
$T_{\pi(z)}M\rightarrow N_{\pi(z)}\mathcal{F}=T_{\pi(z)}M/ T_{\pi(z)}\mathcal{F}$ and the
 tangent map $T_zP \rightarrow T_{\pi(z)}M$ of $\pi$ at $z$.  It was shown in \cite[Page 84]{Mo88} that $\eta$ is a $\mathbb{R}^q$-valued $\mathcal{F}_P$-basic form on $P$. Moreover, the transverse Riemannian metric $g$ induces on $P$ a $o(q)$-valued connection one form $\theta$, called the \emph{transverse Levi-Civita} connection, that is a $o(q)$-valued $\mathcal{F}_P$-basic form on $P$.
 %we define $q+r$ one forms $\eta_1, \cdots, \eta_{q}$ as follows. C 
 Choosing a dual basis $\{e^*_1, \cdots, e^*_r\}$ in $o(q)^*$, define
 \begin{equation}\label{lifted-metric} \alpha_i=<e^*_i, \theta>,\qquad  \forall\, 1\leq i\leq r, \qquad \displaystyle g_P:=\pi_M^* g+ \sum_{i=1}^{r} d\alpha_i\otimes d\alpha_i,\end{equation}
where $\langle \cdot, \cdot\rangle$ is the natural duality pairing between $o(q)^*$ and $o(q)$. By construction, $g_P$ is a transverse Riemannian metric on $P$, and projects to a genuine Riemannian metric $g_W$ on the Molino manifold $W$. 

%The metric $g_P$ further induces a Riemannian metric $g_W$ on $W$ as follows. For any two vector fields $X$ and $Y$ on $W$, 
%lift them to vector fields $\tilde{X}, \tilde{Y}$ on $P$ perpendicular to leaf closures of $\mathcal{F}_P$. 
%At any point $ x\in W$, choose a point $z \in P$ such that $\pi_W(z)=x$, and define 
%\begin{equation}\label{metric2}g_W(X_x, Y_x)= g_P(\tilde{X}_z, \tilde{Y}_z).\end{equation}
%Since $\tilde{X}$ and $\tilde{Y}$ are perpendicular to the leaf closures, they must be perpendicular to leaves as well. Thus
%$g_P(\tilde{X}, \tilde{Y})$ is a basic function relative to both the lifted foliation $\mathcal{F}_P$ and its leaf closures. It follows that the definition of $g_W$ at $x$ does not depend on the choice of $z \in \pi_W^{-1}(x)$, and provides us a Riemannian 
%metric on $W$. 
Now let $\bar{\xi}_M$ be a transverse Killing vector field on $M$, and let
$\xi_M$ be a foliate vector field representing $\bar{\xi}_M$ and preserving
the transverse Riemannian metric $g$. Then the flow $\{\varphi_t\}$ generated
by $\xi_M$ preserves $g$, and hence its tangent map preserves transverse
orthonormal frames for each $t\in \mathbb{R}$. It follows that $\{\varphi_t\}$
lifts to a flow $\{\Phi_t\}$ on $P$ commuting with the action of the structure
group $O(q)$ and preserving the leaves of the lifted foliation
$\mathcal{F}_P$. Let $\xi_P$ denote the induced foliate vector field on $P$.
If $\xi_M$ is everywhere tangent to the leaves of $\mathcal{F}$, then
$\xi_P$ is everywhere tangent to the leaves of $\mathcal{F}_P$. Thus the
transverse vector field on $(P,\mathcal{F}_P)$ determined by $\xi_P$ depends
only on the class $\bar{\xi}_M$. Moreover, it was shown in \cite[Sec. 3.3]{Mo88}
that both $\eta$ and $\theta$ are invariant under $\{\Phi_t\}$, and
consequently the transverse metric $g_P$ is invariant as well. Therefore
the above construction defines a natural lifting homomorphism
\begin{equation}\label{lifting-homo}
\pi_M^{\sharp}:\mathrm{iso}(M, \mathcal{F})\rightarrow \mathrm{iso}(P, \mathcal{F}_P).
\end{equation}

%The transverse tangent sheaf $\mathcal{F}_P$ of $P$ is the sheaf of Lie algebra of transverse vector fields on $M$ associated to the pre-sheaf $V \mapsto \mathfrak{X}( V, \mathcal{F}\vert_V)$, where $V$ ranges over all open subsets of $P$. 

The \emph{Molino's centralizer sheaf} of $P$ is the sheaf $\mathcal{C}_P$ associated to the pre-sheaf 
%consisting of all sections that centralize the global transverse vector fields:
\[ V\mapsto \mathfrak{C}_P(V):=\{v\in \mathfrak{X}(V, \mathcal{F}_P\vert_V)\,\vert\, [v, w]=0\, \text{for all $w\in \mathfrak{X}(P, \mathcal{F}_P)$}\}.
\]
In other words, the sections of $\mathfrak{C}_P(V)$ consist of all local transverse vector fields on $V$ that centralize the global transverse vector fields on $P$. We define the Molino centralizer sheaf $\mathfrak{C}_M$ on $(M, \mathcal{F})$ to be the sheaf associated to the presheaf  
$U \mapsto \mathfrak{C}(U)$ whose sections consists of all transverse Killing vector fields $u$ on $U$ with the property that
$\pi^{\sharp}(u) \in \mathfrak{C}_P(\pi_M^{-1}(U))$, where $U$ ranges over all open subsets of $M$. It is shown in \cite{Mo88} that $\mathfrak{C}_M$ is a locally constant sheaf. More precisely, for any $x\in M$, there exists an open subset $x\in U$, such that the space of sections of $\mathfrak{C}_M(U)$ is an $l$ dimensional real vector space. Moreover, it was shown in \cite{Mo88} that $l$ equals the dimension of the leaf closure of $\mathcal{F}_P$ minus the dimension of the foliation $\mathcal{F}_P$.

We are ready to review Sergiescu's transverse integration \cite{Ser85}. Suppose that $(M, \mathcal{F})$ is a complete Riemannian foliation of codimension $q$. Let $\mathcal{O}_1$ be the orientation line bundle of $TM/T\mathcal{F}$, the normal bundle of the foliation, let $\mathcal{O}_2$ be the orientation line bundle of the flat vector bundle associated to the Molino sheaf $\mathfrak{C}(M, \mathcal{F})$, and let $L=\mathcal{O}_1\otimes \mathcal{O}_2$. Clearly, $L$ is a one dimensional flat vector bundle over $M$. Thus the twisted de Rham differential complex $(\Omega^\bullet(M, L), d)$ is well-defined. As a space, $\Omega^\bullet(M, L)=\Gamma(\wedge^\bullet(T^*M)\otimes L)$.
Choose a family of trivializations $\{(U_{\alpha},\phi_{\alpha})\}$ of the flat line bundle $L$ such
that the corresponding transition functions are locally constant. On each
$U_{\alpha}$, a twisted differential form $\gamma$ has an representation $\gamma= \beta \otimes s$,
where $\beta$ is a differential form on $U_{\alpha}$ and $s$ is a nowhere vanishing local
section of $L$ over $U_{\alpha}$. The exterior derivative of $\gamma$ is given by the formula
\[ d\gamma:= d\beta \otimes s. \]
It is straightforward to check that $d\gamma$ does not depend on the choice of
a local trivialization, and thus defines a global element in $\Omega^{r+1}(M,L)$.
Similarly, given a vector field $X$ on $M$, define
\begin{equation} \label{Lie-derivative} \mathcal{L}_X\gamma = (\mathcal{L}_X\beta )\otimes s, \iota_{X}\gamma= (\iota_X \beta )\otimes s.\end{equation}

One checks easily the above definition does not depend on the choice of
a local trivialization, and extends the usual Lie derivative $\mathcal{L}_X$ and interior
product $\iota_X$ to twisted differential forms. We say that $\gamma \in \Omega^\bullet(M,L)$ is a
twisted basic differential form, if $\mathcal{L}_X\gamma= \iota_X \gamma= 0$, for all $X \in\Gamma(T\mathcal{F})$, and will
denote by $\Omega^\bullet(M,L,\mathcal{F})$ the space of twisted basic differential forms. It is
clear that $\Omega^\bullet(M,L,\mathcal{F})$ is invariant under the exterior derivative defined in
(\ref{Lie-derivative}). We call the cohomology of the differential complex $(\Omega^\bullet(M,L,\mathcal{F}),d)$
the twisted basic cohomology, and denote it by $H^\bullet(M,L, \mathcal{F})$. Moreover, the
same argument as given in the proof \cite[Prop. 7.5]{BT82} shows up to isomorphisms the twisted basic cohomologies are independent of the choice of local trivializations. Similarly, one defines
the complex of twisted compactly supported basic differential forms $(\Omega^\bullet_c(M,L,\mathcal{F}), d)$ and the twisted compactly supported basic de Rham cohomology $H^\bullet_c(M,L, \mathcal{F})$.

Now suppose $\gamma \in \Omega_c^q(M, L, \mathcal{F})$. In a local representation we write
$\gamma = \alpha \otimes s_1 \otimes s_2$, where $\alpha$ is a basic form and $s_1$, $s_2$
are local sections of $\mathcal{O}_1$ and $\mathcal{O}_2$, respectively. In particular,
$s_2$ admits a local expression $s_2 = v_1 \wedge \cdots \wedge v_l$, where
$\{v_1, \ldots, v_l\}$ locally trivializes the Molino sheaf. Since $\mathcal{L}_{v_i}\alpha = 0$
for all $1 \leq i \leq l$, a straightforward computation shows that
$\iota_{s_2}(\alpha \otimes s_1)$ is a twisted form basic with respect to the leaf closures on $P$.
Moreover, $\iota_{s_2}(\alpha \otimes s_1)$ is independent of the choice of local
representation of $\gamma$, and hence descends canonically to a well-defined element
$\gamma_s \in \Omega_c^m(W, L_W)$. Here $m = \dim(W)$, $L_W$ is the orientation line
bundle on $W$, and $\Omega_c^\bullet(W, L_W)$ denotes the space of compactly supported twisted
differential forms on $W$. The transverse integration of $\gamma$ is then defined by
\begin{equation}\label{t-integration} \int_{M/\mathcal{F}}\gamma =\int_W \gamma_s\end{equation}

We refer the interested reader to \cite{LS20} for a detailed exposition on the  properties satisfied by the Sergiescu's transverse integration. Among other things, the usual Stokes' theorem on differentiable manifolds naturally extends to the current setting. We note that for $\alpha \in \Omega^k(M, L, \mathcal{F})$ and $\beta \in \Omega_c^{q-k}(M, L, \mathcal{F})$,
the wedge product $\alpha \wedge \beta \in \Omega_c^q(M, L, \mathcal{F})$, so the transverse
integration $\int_{M/\mathcal{F}} \alpha \wedge \beta$ is well-defined. Moreover, by Stokes'
theorem for the transverse integration defined in (\ref{t-integration}), the pairing
descends to cohomology. 

%When $\mathcal{F}$ is transversally orientable, $L$ is trivial
%and the following reduces to the classical basic Poincar\'e duality.

\begin{thm}[Sergiescu {\cite{Ser85}}]\label{Ser}
Let $\mathcal{F}$ be a complete Riemannian foliation of codimension $q$. The integration pairing
\begin{equation}\label{eq:poincare-pairing}
\langle \cdot\,,\cdot \rangle \colon
H^k(M, \mathcal{F}) \times H_c^{q-k}(M, L, \mathcal{F})
\longrightarrow \mathbb{R},
\end{equation}
defined by
\begin{equation}
\langle\, [\alpha],\, [\beta]\, \rangle
\;=\; \int_{M/\mathcal{F}} \alpha \wedge \beta,
\end{equation}
is a non-degenerate pairing. Consequently, there is a natural isomorphism
\begin{equation}
H^k(M, \mathcal{F}) \;\cong\; \bigl(H_c^{q-k}(M, L, \mathcal{F})\bigr)^*.
\end{equation}
\end{thm}

\section{Transverse Averaging Operators}\label{sec:averaging-operator}
 
The main result of this section is Theorem~\ref{average-operator}, which
constructs a transverse averaging operator for isometric Lie algebra actions on
Riemannian foliations. The overall strategy was explained in the introduction;
here we carry out the construction. We first use Molino theory to produce a
compact Lie group $N$ acting isometrically on the Molino manifold $W$. This
gives rise to an action of $N$ on the basic forms on $(P,\mathcal F_P)$, which
in turn induces, via
\[
\pi_M^*:\Omega^\bullet(M,\mathcal{F})\xrightarrow{\;\sim\;}
\Omega^\bullet(P,\mathcal{F}_P)_{\mathrm{bas}_{O(q)}},
\]
an action on basic forms on $(M,\mathcal F)$. The averaging operator
$\mathcal A$ is then defined from this induced action.

Throughout this section we retain the notation of
Section~\ref{integration}, and assume that $\mathcal{F}$ is a complete
Riemannian foliation on $M$ equipped with a fixed transverse Riemannian
metric $g$, and whose leaf closure space is compact.
 Suppose that we
are given an isometric transverse Lie algebra action
\[
\rho:\mathfrak{g}\to \mathfrak{X}(M,\mathcal{F}).
\]
Let $\xi\in \mathfrak{g}$, and let  $\bar{\xi}_P:=\pi_M^\sharp(\rho(\xi))$ be the image of $\rho(\xi)$ under the lifting homomorphism~(\ref{lifting-homo}). 
One verifies that $\bar{\xi}_P$ projects to a Killing vector field on the Molino manifold
$(W, g_W)$. Thus we obtain
a Lie algebra homomorphism
\[
\mathfrak{g}\to \mathrm{iso}(W,g_W),
\]
where $\mathrm{iso}(W,g_W)$ denotes the Lie algebra of Killing vector fields on
$(W,g_W)$. Let $G$ be the simply connected Lie group with Lie algebra
$\mathfrak{g}$. Since $W$ is compact, every Killing vector field on $W$ is
complete, and therefore the action of $\mathfrak{g}$ on $W$ integrates to an
isometric $G$-action on $(W,g_W)$, giving a homomorphism
\begin{equation}\label{group-action}
G\to \mathrm{Iso}(W,g_W),
\end{equation}
where $\mathrm{Iso}(W,g_W)$ denotes the Riemannian isometry group of $(W,g_W)$. Under our standing assumption, $W$
is compact. Thus $\mathrm{Iso}(W,g_W)$ is a compact Lie group. Let $G_1$ denote the
image of $G$ under \eqref{group-action}, and set
\begin{equation}\label{compact-gp}
N:=\overline{G_1}.
\end{equation}
Then $N$ is a compact and connected Lie group. Note also that the action of $O(q)$ on $P$ is foliate with
respect to the foliation by leaf closures, and therefore descends to an action
of $O(q)$ on $W$. The actions of $N$ and $O(q)$ on $W$ commute.

To proceed, denote by
\[
\Omega^\bullet(P,\mathcal{F}_P)_{\mathrm{bas}_{O(q)}}
\]
the subspace of basic forms on $P$ that are also basic with respect to the
structure group $O(q)$-action on $P$. Then there is a natural isomorphism of
basic de Rham complexes
\begin{equation}\label{pull-back}
\pi_M^*:
(\Omega^\bullet(M,\mathcal{F}),d)\xrightarrow{\;\sim\;}
(\Omega^\bullet(P,\mathcal{F}_P)_{\mathrm{bas}_{O(q)}},d),
\qquad
\beta\mapsto \pi_M^*\beta.
\end{equation}
From now on we denote by
\begin{equation}\label{canonical-iso}
\psi:
(\Omega^\bullet(P,\mathcal{F}_P)_{\mathrm{bas}_{O(q)}},d)
\to
(\Omega^\bullet(M,\mathcal{F}),d)
\end{equation}
the inverse of the isomorphism~(\ref{pull-back}).

Next let $\eta$ be the fundamental one-form on $P$ as in
\eqref{f-one-form}, and choose a dual basis
\[
\{\lambda_1,\dots,\lambda_q\}\subset (\mathbb{R}^q)^*.
\]
Define
\begin{equation}\label{co-frame}
\eta_i=\lambda_i\circ \eta,
\qquad
1\le i\le q.
\end{equation}
It was shown in \cite{Mo88} that if $\phi:M\to M$ is a diffeomorphism
preserving the transverse Riemannian metric $g$ on $M$, then $\phi$ lifts
naturally to a diffeomorphism $\bar{\phi}:P\to P$ preserving the transverse
Riemannian metric $g_P$. Moreover, $\bar{\phi}$ leaves $\eta$ invariant, and
hence preserves each $\eta_i$. This observation plays a crucial role in the
construction of the transverse averaging operator.

Now observe that $\eta_1,\dots,\eta_q$ are $\mathcal{F}_P$-basic and horizontal
with respect to the $O(q)$-action on $P$. Consequently, each form
\[
\gamma\in \Omega^\bullet(P,\mathcal{F}_P)_{\mathrm{bas}_{O(q)}}
\]
admits a unique representation
\begin{equation}\label{basic-representation}
\gamma=
\sum_{1\le i_1<\cdots<i_k\le q}
f_{i_1\cdots i_k}\,
\eta_{i_1}\wedge\cdots\wedge \eta_{i_k}.
\end{equation}
Here the coefficients $f_{i_1\cdots i_k}$ are basic functions with respect to
$\mathcal{F}_P$. Each $f_{i_1\cdots i_k}$ is also basic with respect to the
foliation by leaf closures, and therefore descends to a smooth function
\[
\bar{f}_{i_1\cdots i_k}\in C^\infty(W).
\]

\begin{defn}\label{action-on-forms}
For $h\in N$ and $
\gamma\in \Omega^\bullet(P,\mathcal{F}_P)_{\mathrm{bas}_{O(q)}}$ 
with representation \eqref{basic-representation}, define
\begin{equation}\label{action-formula}
h^*\gamma=
\sum_{1\le i_1<\cdots<i_k\le q}
\pi_W^*(h^*\bar{f}_{i_1\cdots i_k})\,
\eta_{i_1}\wedge\cdots\wedge \eta_{i_k}.
\end{equation}
This defines an action of the compact and connected Lie group $N$ on
$\Omega^\bullet(P,\mathcal{F}_P)_{\mathrm{bas}_{O(q)}}$. 

Through the isomorphism of cochain complexes
\[
\psi:
\Omega^\bullet(P,\mathcal{F}_P)_{\mathrm{bas}_{O(q)}}
\to
\Omega^\bullet(M,\mathcal{F}),
\]
this further induces an $N$ action on $\Omega^\bullet(M,\mathcal{F})$ by

\begin{equation}\label{action-formula2}
h^*\beta:=\psi\bigl(h^*\pi_M^*\beta\bigr), 
\end{equation}
where $h\in N$ and $\beta \in \Omega^\bullet(M,\mathcal{F})$.
\end{defn}

The following lemma shows that \eqref{action-formula} defines an action of
$N$ on $\Omega^\bullet(P,\mathcal{F}_P)_{\mathrm{bas}_{O(q)}}$ by cochain
automorphisms. Consequently, the induced action of $N$ on
$\Omega^\bullet(M,\mathcal{F})$ also commutes with the exterior differential.

\begin{lem}\label{action-property}
Let
\[
\gamma\in \Omega^\bullet(P,\mathcal{F}_P)_{\mathrm{bas}_{O(q)}}
\]
have the representation \eqref{basic-representation}, and let $h\in N$. Then:
\begin{itemize}
\item[(a)] The form $h^*\gamma$ defined by \eqref{action-formula} lies in
$\Omega^\bullet(P,\mathcal{F}_P)_{\mathrm{bas}_{O(q)}}$.
\item[(b)] One has
\[
d(h^*\gamma)=h^*(d\gamma).
\]
\end{itemize}
\end{lem}

\begin{proof}
To prove~(a), it suffices to show that $h^*\gamma$ is $O(q)$-invariant, since by construction it is clearly $\mathcal{F}_P$-basic and $O(q)$-horizontal. We therefore assume that $\gamma$ is written as in \eqref{basic-representation}.

Let $\xi\in \mathfrak{g}$, let $\xi_M$ be a foliate vector field representing
$\rho(\xi)$ and preserving the transverse Riemannian metric $g$ on $M$, and let $\xi_P$ be its natural lift to $P$. Denote by
$\{\Phi_t\}$ the flow generated by $\xi_P$ on $P$, and by
$\{\bar{\Phi}_t\}$ the induced flow on $W$. By construction,
\[
\pi_W^*(\bar{\Phi}_t^*\bar{f}_{i_1\cdots i_k})
=
\Phi_t^*f_{i_1\cdots i_k}.
\]
Since each $\eta_i$ is invariant under $\Phi_t^*$, it follows that
\[
\sum_{i_1<\cdots<i_k}
\pi_W^*(\bar{\Phi}_t^*\bar{f}_{i_1\cdots i_k})\,
\eta_{i_1}\wedge\cdots\wedge \eta_{i_k}
=
\Phi_t^*
\Bigl(
\sum_{i_1<\cdots<i_k}
f_{i_1\cdots i_k}\,
\eta_{i_1}\wedge\cdots\wedge \eta_{i_k}
\Bigr).
\]
In other words, in the notation of \eqref{action-formula},
\begin{equation}\label{invariance}
\bar{\Phi}_t^*\gamma=\Phi_t^*\gamma.
\end{equation}
Since $\Phi_t$ commutes with the action of the structure group $O(q)$ and
$\gamma$ is $O(q)$-basic, the right-hand side of \eqref{invariance} is
$O(q)$-invariant. Therefore $\bar{\Phi}_t^*\gamma$ is $O(q)$-invariant.

This shows that $h^*\gamma$ is $O(q)$-invariant for every $h$ belonging to a
one-parameter subgroup of $G_1$. Since $G_1$ is connected, every element of
$G_1$ is a finite product of elements lying in one-parameter subgroups of
$G_1$, and hence $h^*\gamma$ is $O(q)$-invariant for every $h\in G_1$. Since
$O(q)$-invariance is a closed condition and $N=\overline{G_1}$, it follows by
continuity that $h^*\gamma$ is $O(q)$-invariant for every $h\in N$. This proves~(a).

For every $h$ belonging to a one-parameter subgroup of $G_1$, the identity
\begin{equation}\label{commutes-d}
d(h^*\gamma)=h^*(d\gamma)
\end{equation}
follows from \eqref{invariance} and the fact that exterior differentiation
commutes with pullback. Since $G_1$ is connected, every element of $G_1$ can be
written as a finite product of elements lying in one-parameter subgroups of
$G_1$. Hence \eqref{commutes-d} holds for every $h\in G_1$. The same identity
for all $h\in N$ then follows by continuity. This proves~(b).
\end{proof}

We now define the transverse averaging operator.

\begin{defn}\label{t-a-operator}
For every $\beta\in \Omega^\bullet(M,\mathcal{F})$, define
\begin{equation}\label{a-operator}
\mathcal{A}(\beta)=\int_N h^*\beta\,dh,
\end{equation}
where $dh$ denotes the normalized left-invariant Haar measure on the compact
Lie group $N$, so that
\[
\int_N 1\,dh=1.
\]
\end{defn}

\begin{lem}\label{lem:avaraging-commutes-d}
The averaging operator $\mathcal{A}$ commutes with the exterior differential $d$,
i.e., $d\mathcal{A}(\beta) = \mathcal{A}(d\beta)$ for all $\beta \in \Omega^\bullet(M,\mathcal{F})$.
\end{lem}

\begin{proof}
Lemma \ref{action-property} (b), together with the fact that $\psi$ is an isomorphism of cochain complexes, implies that $d(h^*\beta) = h^*(d\beta)$ for every $h \in N$.
Since $d$ is a continuous linear operator on $\Omega^\bullet(M,\mathcal{F})$, it commutes
with the integral against the Haar measure on $N$. Therefore
\[
d\mathcal{A}(\beta) = d\int_N h^*\beta\, dh
= \int_N d(h^*\beta)\, dh
= \int_N h^*(d\beta)\, dh
= \mathcal{A}(d\beta). \qedhere
\]
\end{proof}
%Now define 
%\begin{equation} \label{averaging-complex} A^i(M, \mathcal{F})=\mathcal{A}(\Omega^i(M, \mathcal{F})).\end{equation}

%It follows easily from Lemma \ref{action-property} that $A^i(M, \mathcal{F})\subset \Omega^i(M, \mathcal{F})$; moreover, Lemma \ref{lem:avaraging-commutes-d} implies that $d(A^i(M, \mathcal{F})) \subset A^{i+1}(M, \mathcal{F})$. Thus $(A^*(M, \mathcal{F}), d)$ is a differential subcomplex of $(\Omega(M, \mathcal{F}), d)$.

\begin{lem}\label{chain-homotopy}  Let $h \in G_1$, and let $\beta \in \Omega^k(M,\mathcal{F})$ be a closed basic form. We have that
\[ [h^*\beta]=[\beta] \in H^{k}(M, \mathcal{F}).\]
\end{lem}

\begin{proof} It suffices to show that for a closed basic form $\gamma \in \Omega^k(P, \mathcal{F}_P)_{\text{bas}_{O(q)}}$, $[h^*\gamma]=[\gamma]\in H^k(\Omega(P, \mathcal{F}_P)_{\text{bas}_{O(q)}}, d)$. We retain the notation used in the proof of Lemma \ref{action-property}. Let  $\xi \in \mathfrak{g}$, and let 
$\{\bar{\Phi}_t\}$ be the flow generated by the fundamental vector field on $W$ induced by $\xi$, and let $\{\Phi_t\}$ be the flow induced by a foliate vector field $\xi_P$ on $P$ representing $\pi^{\sharp}_M(\xi)$.  It was shown in the proof of Lemma \ref{action-property} that $\bar{\Phi}_t^*\gamma =\Phi_t^*\gamma$. Therefore
\[ \bar{\Phi}_t^*\gamma -\gamma =\Phi_t^*\gamma -\gamma
= \int_0^t\, \dfrac{d \Phi_s^*\gamma}{ds}
= \int_0^t\, \Phi^*_s\mathcal{L}_{\xi_P} \gamma 
=\int_0^t\, \Phi^*_s(d\iota_{\xi_P}+ \iota_{\xi_P} d)\gamma
=d\int_0^t\, \Phi^*_s\iota_{\xi_P}\gamma.
\]
Here we have used the fact that $\gamma$ is a closed form. Since $\xi_P$ preserves the foliation $\mathcal{F}_P$ and is invariant under the $O(q)$ action on $P$, the form $\iota_{\xi_P}\gamma\in \Omega(P, \mathcal{F}_P)_{\text{bas}_{O(q)}}$. Hence $\Phi^*_s\iota_{\xi_P}\gamma \in \Omega(P, \mathcal{F}_P)_{\text{bas}_{O(q)}}$  for every $s$. The assertion of Lemma \ref{chain-homotopy} now follows from the fact that $G_1$ is connected, and every $h\in G_1$
is a finite product of elements belonging to a one-parameter subgroup of $G_1$.

\end{proof}

%Choose a closed test form $\beta \in \Omega^{q-k}_c(M, L, \mathcal{F})$.

The extension from $G_1$ to all of $N$ is carried out in the proof
of Theorem~\ref{average-operator} below. The key point is that for any given
$\beta \in \Omega^\bullet(M, \mathcal{F})$ and
$\gamma \in \Omega^\bullet_c(M, L,\mathcal{F})$, the function
\begin{equation}\label{continuity}
N \to \mathbb{R}, \quad h \mapsto \int_{M/\mathcal{F}} h^*\beta \wedge \gamma
\end{equation}
is continuous on $N$ and constant on the dense subgroup $G_1$, hence constant
on all of $N$. We are now ready to prove Theorem~\ref{average-operator}.

\begin{proof}
Clearly, for all $\beta \in \Omega^k(M, \mathcal{F})$ and all $h\in N$, we have
$h^*\mathcal{A}(\beta)=\mathcal{A}(\beta)$. It follows that
$\mathcal{A}(\beta) \in \Omega^k(M, \mathcal{F})^{\mathfrak{g}}$. Next we show that,
for every closed basic form $\beta\in \Omega^{k}(M, \mathcal{F})$ and every $h\in N$,
one has $[h^*\beta]=[\beta]$.

Consider the function \eqref{continuity}, where $\beta\in \Omega(M, \mathcal{F})$
is a closed basic form representing a class in $H^k(M,\mathcal F)$, and
$\gamma \in \Omega_c^\bullet(M, L, \mathcal{F})$ is a closed twisted basic form
representing a class in $H^{q-k}_c(M, L, \mathcal{F})$. We first verify that the map
\eqref{continuity} is continuous. Since $\gamma$ is compactly supported,
$h^*\beta \wedge \gamma \in \Omega_c^q(M, L, \mathcal{F})$, and by \eqref{t-integration}
the map \eqref{continuity} is equal to
$h \mapsto \int_W (h^*\beta \wedge \gamma)_s$.
By Definition~\ref{action-on-forms}, the action of $N$ on $\beta$ is determined by the pullback of
globally defined smooth functions on the compact Molino manifold $W$. Since $N$
acts smoothly on $W$, these functions converge uniformly on $W$ as $h \to h_0$ in $N$.
The dominated convergence theorem therefore gives
\[
\int_W (h^*\beta \wedge \gamma)_s \longrightarrow \int_W (h_0^*\beta \wedge \gamma)_s,
\]
establishing continuity.

By Lemma~\ref{chain-homotopy}, the function \eqref{continuity} is constant
on the dense subgroup $G_1$ of $N$. Hence it is constant on all of $N$. In other
words, for all $h\in N$ we have
\[
\int_{M/\mathcal{F}} h^*\beta \wedge \gamma
=
\int_{M/\mathcal{F}} \beta\wedge \gamma.
\]

Therefore
\[
\begin{split}
\int_{M/\mathcal{F}} \mathcal{A}(\beta)\wedge \gamma
&=\int_{M/\mathcal{F}} \left(\int_N h^*\beta\, dh\right) \wedge \gamma \\
&=\int_N \left(\int_{M/\mathcal{F}} h^*\beta \wedge \gamma\right) dh \\
&=\int_N \left(\int_{M/\mathcal{F}} \beta \wedge \gamma\right) dh \\
&=\int_{M/\mathcal{F}} \beta \wedge \gamma,
\end{split}
\]
where the second equality follows from Fubini's theorem.

Since $[\gamma]\in H_c^{q-k}(M, L, \mathcal{F})$ is arbitrary, Theorem~\ref{Ser}
implies that $[\mathcal{A}(\beta)]=[\beta]$.
\end{proof}

We are ready to prove Theorem \ref{thm:E1}. The key difficulty in the proof is the existence of an averaging operator on the space of polynomial-valued basic differential forms. Once this is established, the proof proceeds via a short exact sequence argument, which we learned from Sjamaar in the context of invariant $G$-complexes in private communications.

\begin{proof}[Proof of Theorem~\ref{thm:E1}]

We retain the notation introduced in Section \ref{sec:background}. Let $C^{ij}:= S^i\mathfrak{g}^* \otimes \Omega^{j-i}(M, \mathcal{F})$, let $Z^{ij}\subset C^{ij}$ be the kernel of the vertical differential $d$, and let $B^{ij}= d(C^{ij-i})$.  Note that elements in $C^{\bullet\bullet}$ may be naturally identified with a polynomial-valued basic differential form on the foliated manifold $(M,\mathcal{F})$. The transverse averaging operator $\mathcal{A}: \Omega^\bullet(M, \mathcal{F})\rightarrow \Omega^\bullet(M, \mathcal{F})^{\mathfrak{g}}$ naturally extends to an equivariant transverse averaging operator 
\begin{equation}\label{e-t-a-operator} \mathcal{A}_{\mathfrak{g}}: C^{ij} \rightarrow \Omega_{\mathfrak{g}}^{ij}(M, \mathcal{F})\end{equation}  which is a projection. 

Let $Z^{ij}_{\mathfrak{g}}\subset  \Omega^{ij}_{\mathfrak{g}}(M, \mathcal{F})$ be the kernel of the vertical differential, and let 
 $B^{ij}_{\mathfrak{g}}= d(\Omega^{ij}_{\mathfrak{g}}(M, \mathcal{F}))$. Since the equivariant transverse averaging operator (\ref{e-t-a-operator}) is a chain map that commutes with $d$, we have that $Z^{ij}_{\mathfrak{g}}=(Z^{ij})^{\mathfrak{g}}$ and that $B^{ij}_{\mathfrak{g}}= (B^{ij})^{\mathfrak{g}}$. By definition, the $E_1$-term of the equivariant basic Cartan complex is given by 
 \[ E_1^{ij}=Z^{ij}_{\mathfrak{g}}/B^{ij}_{\mathfrak{g}}= (Z^{ij})^{\mathfrak{g}}/(B^{ij})^{\mathfrak{g}}.\]
 The equivariant transverse averaging operator (\ref{e-t-a-operator}), together with the short exact sequence
\[ 0\hookrightarrow B^{\bullet\bullet} \rightarrow Z^{\bullet\bullet}\rightarrow H^{\bullet\bullet}(C, d)\rightarrow 0,\]
gives rise to a short exact sequence
\[ 0\hookrightarrow (B^{\bullet\bullet})^{\mathfrak{g}}\rightarrow (Z^{\bullet\bullet})^{\mathfrak{g}}\rightarrow \left( H^{\bullet\bullet}(C, d)\right)^{\mathfrak{g}}\rightarrow 0.\]
Note that $H^{\bullet\bullet}(C, d)=S^\bullet\mathfrak{g}^*\otimes H^\bullet(M,\mathcal{F})$, and that the Cartan formula implies that
$\mathfrak{g}$ acts on $H^\bullet(M,\mathcal{F})$ trivially. This gives
\[ E_1^{\bullet\bullet} = \left(H^{\bullet\bullet}(C, d)\right)^{\mathfrak{g}}
\cong \left(S^\bullet\mathfrak{g}^* \otimes H^{\bullet}(M,
\mathcal{F})\right)^{\mathfrak{g}}= \left(S^\bullet\mathfrak{g}^*\right)^{\mathfrak{g}}\otimes H^\bullet(M, \mathcal{F}),\]
completing the proof.

\end{proof}

\section{Basic cohomology of Lie foliations of compact type}
\label{App:Lie-foliation}

This section proves Theorem~\ref{thm:liefoliation}, which identifies the
basic cohomology of a complete Lie $\mathfrak g$-foliation of compact type
with the Lie algebra cohomology $H^\bullet(\mathfrak g)$. We begin by
recalling the definition of a Lie foliation, following Molino~\cite{Mo88}.

\begin{defn}\label{lie-g-foliation}
A \emph{transverse parallelism} of $(M,\mathcal F)$ by $\mathfrak g$ is a
$\mathfrak g$-valued $1$-form
$\alpha\in\Omega^1(M,\mathfrak g)$ satisfying:
\begin{enumerate}
\item[(i)] for every $x\in M$, the map
$\alpha_x:T_xM\to\mathfrak g$ is surjective;
\item[(ii)] $\ker\alpha_x=T_x\mathcal F$;
\item[(iii)] the Maurer--Cartan equation
\begin{equation}\label{MC-eq}
d\alpha+\tfrac12[\alpha,\alpha]=0
\end{equation}
holds.
\end{enumerate}
The form $\alpha$ is called the \emph{Maurer--Cartan form}. A foliation
equipped with such a form is called a \emph{Lie $\mathfrak g$-foliation},
or simply a \emph{Lie foliation} with structural Lie algebra $\mathfrak g$.
\end{defn}

Let $(M,\mathcal F)$ be a complete Lie $\mathfrak g$-foliation with
Maurer--Cartan form $\alpha$. Since $\alpha$ identifies the normal bundle
$N\mathcal F$ with the trivial bundle $M\times\mathfrak g$, for each
$\xi\in\mathfrak g$ there is a unique transverse vector field
$\overline X_\xi$ satisfying
\[
\alpha(\overline X_\xi)=\xi.
\]
The Maurer--Cartan equation gives
\[
[\overline X_\xi,\overline X_\eta]
=
\overline X_{[\xi,\eta]}.
\]
Consequently,
\begin{equation}\label{canonical-action}
\rho:\mathfrak g\longrightarrow\mathfrak X(M,\mathcal F),
\qquad
\xi\longmapsto\overline X_\xi,
\end{equation}
is a Lie algebra homomorphism and hence defines a canonical transverse
$\mathfrak g$-action.

An inner product $\langle\cdot,\cdot\rangle$ on $\mathfrak g$ induces a
transverse Riemannian metric by
\begin{equation}\label{in-metric}
g(u,v)
:=
\langle\alpha(u),\alpha(v)\rangle,
\qquad
u,v\in N_x\mathcal F.
\end{equation}

\begin{lem}\label{isometric-action}
If the inner product in \eqref{in-metric} is
$\operatorname{ad}$-invariant, then the canonical transverse
$\mathfrak g$-action \eqref{canonical-action} is isometric.
\end{lem}

\begin{proof}
Let $e_1,\dots,e_q$ be an orthonormal basis of $\mathfrak g$, and set
$\overline X_i=\rho(e_i)$. The transverse vector fields
$\overline X_1,\dots,\overline X_q$ form a global orthonormal frame of
$N\mathcal F$.

For $\xi\in\mathfrak g$, set $\overline X=\rho(\xi)$. Since
$\rho$ is a Lie algebra homomorphism,
\[
[\overline X,\overline X_i]=\rho([\xi,e_i]).
\]
It follows that
\[
\begin{split}
(\mathcal L_{\overline X}g)(\overline X_i,\overline X_j)
&=
-g([\overline X,\overline X_i],\overline X_j)
-g(\overline X_i,[\overline X,\overline X_j])\\
&=
-\langle[\xi,e_i],e_j\rangle
-\langle e_i,[\xi,e_j]\rangle
=0.
\end{split}
\]
Thus $\mathcal L_{\overline X}g=0$ for every $\xi\in\mathfrak g$.
\end{proof}

\subsection{A direct transverse averaging operator}

We now assume the hypotheses of Theorem~\ref{thm:liefoliation}: $\mathcal F$ is a complete Lie-$\mathfrak g$ foliation on a connected manifold $M$, $\mathfrak g$ is of compact type, and the leaf closure space
$
B:=M/\overline{\mathcal F}$
is compact. By the structure theorem for complete Lie foliations, the leaf
closures form a regular foliation. Consequently, $B$ is a compact connected
manifold, and the quotient map
$
p:M\longrightarrow B$
is a smooth submersion.

The transverse vector fields $\overline X_\xi$ preserve the leaf closures
and therefore project to vector fields $\xi_B$ on $B$.

\begin{lem}\label{coefficient-complex}
The following statements hold.
\begin{itemize}
\item[(a)] For every $k$, the assignment
\begin{equation}\label{coefficient-map}
F_\beta(p(x))(\xi_1,\dots,\xi_k)
:=
\beta_x(\overline X_{\xi_1},\dots,\overline X_{\xi_k})
\end{equation}
defines an isomorphism
\begin{equation}\label{coefficient-isomorphism}
\mathcal I_k:
\Omega^k(M,\mathcal F)
\xrightarrow{\;\cong\;}
C^\infty(B,\wedge^k\mathfrak g^*),
\qquad
\beta\longmapsto F_\beta.
\end{equation}

\item[(b)] Under \eqref{coefficient-isomorphism}, the exterior differential
corresponds to the operator
$
D:C^\infty(B,\wedge^k\mathfrak g^*)
\longrightarrow
C^\infty(B,\wedge^{k+1}\mathfrak g^*)$
given by
\begin{align}
(DF)(\xi_0,\dots,\xi_k)
={}&
\sum_{i=0}^k(-1)^i(\xi_i)_B
\bigl(
F(\xi_0,\dots,\widehat{\xi_i},\dots,\xi_k)
\bigr)
\nonumber\\
&+
\sum_{0\leq i<j\leq k}(-1)^{i+j}
F\bigl(
[\xi_i,\xi_j],
\xi_0,\dots,\widehat{\xi_i},\dots,
\widehat{\xi_j},\dots,\xi_k
\bigr).
\label{coefficient-differential}
\end{align}

\item[(c)] For every $\xi\in\mathfrak g$ and
$\beta\in\Omega^k(M,\mathcal F)$,
\begin{equation}\label{coefficient-lie-derivative}
F_{\mathcal L_{\rho(\xi)}\beta}
=
\xi_BF_\beta+\operatorname{ad}^*(\xi)F_\beta,
\end{equation}
where $\xi_BF_\beta$ denotes componentwise differentiation of the
$\wedge^k\mathfrak g^*$-valued function $F_\beta$ along $\xi_B$.
\end{itemize}
\end{lem}

\begin{proof}
For fixed $\xi_1,\dots,\xi_k$, the right-hand side of
\eqref{coefficient-map} is a basic function on $M$ and is therefore
constant on each leaf closure. Thus $F_\beta$ is well defined.

Conversely, if
$F\in C^\infty(B,\wedge^k\mathfrak g^*)$, define
\[
(\mathcal I_k^{-1}F)_x(v_1,\dots,v_k)
=
F(p(x))
\bigl(
\alpha_x(v_1),\dots,\alpha_x(v_k)
\bigr).
\]
This is a basic form, and the two constructions are inverse to each other.
This proves~(a).

Part~(b) follows from the formula for the exterior differential and the
identity
\[
[\overline X_\xi,\overline X_\eta]
=
\overline X_{[\xi,\eta]}.
\]
Indeed, evaluating $d\beta$ on
$\overline X_{\xi_0},\dots,\overline X_{\xi_k}$ gives precisely
\eqref{coefficient-differential}.

For~(c), fix $\xi,\xi_1,\dots,\xi_k\in\mathfrak g$, and choose foliate
representatives $X$ and $Y_i$ of $\overline X_\xi$ and
$\overline X_{\xi_i}$, respectively. The function
$\beta(Y_1,\dots,Y_k)$ is basic and descends to
$F_\beta(\xi_1,\dots,\xi_k)$ on $B$. Moreover, $X$ projects to $\xi_B$,
while the transverse class of $[X,Y_i]$ is
$\overline X_{[\xi,\xi_i]}$. Using the Leibniz rule for the Lie derivative,
we obtain
\begin{align*}
F_{\mathcal L_{\rho(\xi)}\beta}(\xi_1,\dots,\xi_k)
&=
X\bigl(\beta(Y_1,\dots,Y_k)\bigr)
-\sum_{i=1}^k
\beta(Y_1,\dots,[X,Y_i],\dots,Y_k)\\
&=
\xi_B\bigl(F_\beta(\xi_1,\dots,\xi_k)\bigr)
-\sum_{i=1}^k
F_\beta(\xi_1,\dots,[\xi,\xi_i],\dots,\xi_k)\\
&=
\bigl(\xi_BF_\beta+\operatorname{ad}^*(\xi)F_\beta\bigr)
(\xi_1,\dots,\xi_k).
\end{align*}
This proves~(c).
\end{proof}
Choose an $\operatorname{ad}$-invariant positive-definite inner product on
$\mathfrak g$. By Lemma~\ref{isometric-action}, the projected vector fields
$\xi_B$ are Killing vector fields. Since $B$ is compact, they are complete
and integrate to a right isometric action of the simply connected Lie group
$\widetilde G$ with Lie algebra $\mathfrak g$. We write this action as
\[
(b,a)\longmapsto b\cdot a
\]
and set $R_a(b)=b\cdot a$. With this convention,
\begin{equation}\label{projected-transformation}
(R_a)_*(\xi_B)
=
(\operatorname{Ad}_{a^{-1}}\xi)_B.
\end{equation}

The inner product on $\mathfrak g$ is
$\operatorname{Ad}$-invariant, so
$\operatorname{Ad}(\widetilde G)\subset O(\mathfrak g)$. The map
\[
\jmath:\widetilde G
\longrightarrow
\operatorname{Iso}(B)\times O(\mathfrak g),
\qquad
\jmath(a):=(R_{a^{-1}},\operatorname{Ad}_a),
\]
is a group homomorphism. Define
\begin{equation}\label{graph-closure}
\widehat N
:=
\overline{\jmath(\widetilde G)}
\subset
\operatorname{Iso}(B)\times O(\mathfrak g).
\end{equation}
Then $\widehat N$ is a compact connected Lie group. We take the closure of the graph rather than merely the closure of the image
of $\widetilde G$ in $\operatorname{Iso}(B)$. Indeed, an element of $\widetilde G$ acting
trivially on $B$ need not act trivially on $\mathfrak g$ through the adjoint
representation. Thus the induced isometry of $B$ does not, by itself,
determine the corresponding action on the coefficient space
$\wedge^\bullet\mathfrak g^*$.

Every $(\phi,A)\in\widehat N$ satisfies
\begin{equation}\label{graph-relations}
A\in\overline{\operatorname{Ad}(\widetilde G)}
\subset\operatorname{Aut}(\mathfrak g)\cap O(\mathfrak g),
\qquad
\phi_*(\xi_B)=(A\xi)_B.
\end{equation}
Indeed, these relations hold on the dense subgroup
$\jmath(\widetilde G)$ and hence extend to its closure.

For $n=(\phi,A)\in\widehat N$ and
$F\in C^\infty(B,\wedge^\bullet\mathfrak g^*)$, define
\begin{equation}\label{graph-action}
(n\cdot F)(b)
:=
(A^{-1})^*F\bigl(\phi^{-1}(b)\bigr).
\end{equation}

\begin{lem}\label{graph-action-lemma}
The action \eqref{graph-action} commutes with $D$. Under
\eqref{coefficient-isomorphism}, its restriction to
$\jmath(\widetilde G)$ integrates the canonical transverse
$\mathfrak g$-action.
\end{lem}

\begin{proof}
For $a\in\widetilde G$, the action is given by
\[
(\jmath(a)\cdot F)(b)
=
(\operatorname{Ad}_{a^{-1}})^*F(b\cdot a).
\]
Fix $\xi\in\mathfrak g$, and let
$\exp:\mathfrak g\to\widetilde G$ denote the exponential map. Differentiating
the action of the one-parameter subgroup $t\mapsto\exp(t\xi)$ gives
\[
\left.\frac{d}{dt}\right|_{t=0}
\jmath(\exp(t\xi))\cdot F
=
\xi_BF+\operatorname{ad}^*(\xi)F.
\]
By Lemma~\ref{coefficient-complex}(c), this infinitesimal action corresponds
under \eqref{coefficient-isomorphism} to
$\mathcal L_{\rho(\xi)}$ on basic forms. It therefore commutes with the
exterior differential. Since $\widetilde G$ is connected and generated by
its one-parameter subgroups, the action of $\jmath(\widetilde G)$ commutes
with $D$ and integrates the canonical transverse $\mathfrak g$-action. The
same conclusion holds for $\widehat N$ by density and continuity.
\end{proof}

Using \eqref{coefficient-isomorphism}, define the corresponding action on
basic forms by
\[
n\cdot\beta
:=
\mathcal I^{-1}(n\cdot F_\beta).
\]

\begin{prop}\label{direct-lie-averaging}
Define
\begin{equation}\label{direct-average}
\mathcal A_{\mathrm{Lie}}(\beta)
:=
\int_{\widehat N}n\cdot\beta\,dn,
\end{equation}
where $dn$ is the normalized Haar measure on $\widehat N$. Then
\[
\mathcal A_{\mathrm{Lie}}:
\Omega^\bullet(M,\mathcal F)
\longrightarrow
\Omega^\bullet(M,\mathcal F)^{\mathfrak g}
\]
is a cochain projection. Moreover, if $\beta$ is closed, then
\[
[\mathcal A_{\mathrm{Lie}}(\beta)]
=
[\beta]
\quad\text{in }H^\bullet(M,\mathcal F).
\]
\end{prop}

\begin{proof}
By Lemma~\ref{graph-action-lemma}, every $n\in\widehat N$ acts by cochain
automorphisms, so $\mathcal A_{\mathrm{Lie}}$ commutes with $d$. For
$m\in\widehat N$, left invariance of the Haar measure gives
\[
m\cdot\mathcal A_{\mathrm{Lie}}(\beta)
=
\int_{\widehat N}(mn)\cdot\beta\,dn
=
\mathcal A_{\mathrm{Lie}}(\beta).
\]
Thus the image of $\mathcal A_{\mathrm{Lie}}$ is contained in the
$\widehat N$-invariant subcomplex. Conversely, every
$\mathfrak g$-invariant basic form is fixed by the connected group
$\jmath(\widetilde G)$ and hence, by continuity, by its closure
$\widehat N$. Therefore $\mathcal A_{\mathrm{Lie}}$ is a cochain projection
onto $\Omega^\bullet(M,\mathcal F)^{\mathfrak g}$.

The cohomological assertion follows by the same argument as in
Theorem~\ref{average-operator}. For completeness, if
$a=\exp(t\xi)$ and $\beta$ is closed, Cartan's formula gives
\[
\jmath(a)\cdot\beta-\beta
=
d\int_0^t
\jmath(\exp(s\xi))\cdot\iota(\xi)\beta\,ds.
\]
Since $\widetilde G$ is connected, every element of
$\jmath(\widetilde G)$ acts trivially on basic cohomology.

For a closed compactly supported twisted basic form $\gamma$, the function
\[
\widehat N\longrightarrow\mathbb R,
\qquad
n\longmapsto
\int_{M/\mathcal F}(n\cdot\beta)\wedge\gamma
\]
is continuous and constant on the dense subgroup
$\jmath(\widetilde G)$. It is therefore constant on $\widehat N$.
Sergiescu's basic Poincar\'e duality gives
$[n\cdot\beta]=[\beta]$ for every $n\in\widehat N$. Averaging over
$\widehat N$ proves the assertion.
\end{proof}

\subsection{The invariant subcomplex}

Since the vector fields $\xi_B$ span $TB$, the $\widetilde G$-action on
$B$ is transitive. Hence the action of $\widehat N$ on $B$ through its first
component,
$
(\phi,A)\cdot b:=\phi(b)$, 
is also transitive. Fix $b_0\in B$, and let
\[
\widehat K
:=
\{(\phi,A)\in\widehat N\mid\phi(b_0)=b_0\}
\]
be the stabilizer of $b_0$ under this action. The group $\widehat K$ acts on
$\wedge^\bullet\mathfrak g^*$ by
\[
(\phi,A)\cdot\tau=(A^{-1})^*\tau.
\]

\begin{lem}\label{invariant-coefficient-complex}
Evaluation at $b_0$ induces an isomorphism
\[
\operatorname{ev}_{b_0}:
C^\infty(B,\wedge^k\mathfrak g^*)^{\widehat N}
\xrightarrow{\;\cong\;}
(\wedge^k\mathfrak g^*)^{\widehat K}.
\]
For every $\widehat N$-invariant coefficient map $F$,
\[
\operatorname{ev}_{b_0}(DF)
=
-d_{CE}\operatorname{ev}_{b_0}(F).
\]
Consequently, in degree $k$, the map
$
\beta\longmapsto(-1)^kF_\beta(b_0)$ 
defines an isomorphism of cochain complexes
\[
\Omega^\bullet(M,\mathcal F)^{\mathfrak g}
\xrightarrow{\;\cong\;}
(\wedge^\bullet\mathfrak g^*)^{\widehat K}.
\]
\end{lem}

\begin{proof}
Transitivity shows that an $\widehat N$-invariant coefficient map is
determined by its value at $b_0$, and this value must be fixed by
$\widehat K$.

Conversely, if
$\tau\in(\wedge^k\mathfrak g^*)^{\widehat K}$, define
\[
F_\tau(\phi(b_0))
:=
(A^{-1})^*\tau,
\qquad
(\phi,A)\in\widehat N.
\]
The $\widehat K$-invariance of $\tau$ makes this definition independent of
the choice of $(\phi,A)$. Since $B\cong\widehat N/\widehat K$, the resulting
map is smooth and gives the inverse of evaluation.

If $F$ is $\widehat N$-invariant, differentiation of its invariance under
$\jmath(\widetilde G)$ gives
\[
\xi_BF=-\operatorname{ad}^*(\xi)F.
\]
Substituting this identity into \eqref{coefficient-differential} yields
\[
(DF)(b_0)
=
-d_{CE}\bigl(F(b_0)\bigr).
\]
The factor $(-1)^k$ gives the stated cochain isomorphism.
\end{proof}

\begin{lem}\label{stabilizer-quasi-isomorphism}
The inclusion
\[
(\wedge^\bullet\mathfrak g^*)^{\widehat K}
\hookrightarrow
\wedge^\bullet\mathfrak g^*
\]
is a quasi-isomorphism.
\end{lem}

\begin{proof}
For every $(\phi,A)\in\widehat K$, one has
$A\in\overline{\operatorname{Ad}(\widetilde G)}$. Inner automorphisms act
trivially on $H^\bullet(\mathfrak g)$ by Cartan's homotopy formula. Since the
induced representation on $H^\bullet(\mathfrak g)$ is continuous, its kernel
is closed. It follows that every
$A\in\overline{\operatorname{Ad}(\widetilde G)}$ also acts trivially on
$H^\bullet(\mathfrak g)$.

Define
\[
P_{\widehat K}(\tau)
:=
\int_{\widehat K}(A^{-1})^*\tau\,dn.
\]
This is a cochain projection onto
$(\wedge^\bullet\mathfrak g^*)^{\widehat K}$ and induces the identity on
$H^\bullet(\mathfrak g)$. Hence every cohomology class has an
$\widehat K$-invariant representative.

If an invariant cocycle $\tau$ is exact in
$\wedge^\bullet\mathfrak g^*$, say $\tau=d_{CE}\nu$, then
\[
\tau
=
P_{\widehat K}(\tau)
=
d_{CE}P_{\widehat K}(\nu).
\]
Thus $\tau$ is already exact in the invariant subcomplex.
\end{proof}

We can now complete the proof of Theorem~\ref{thm:liefoliation}.

\begin{proof}
Proposition~\ref{direct-lie-averaging} implies that the inclusion
\[
\Omega^\bullet(M,\mathcal F)^{\mathfrak g}
\hookrightarrow
\Omega^\bullet(M,\mathcal F)
\]
is a quasi-isomorphism. Indeed, surjectivity in cohomology follows by
averaging closed representatives. If an invariant cocycle $\beta$ satisfies
$\beta=d\gamma$, then
$
\beta
=
\mathcal A_{\mathrm{Lie}}(\beta)
=
d\,\mathcal A_{\mathrm{Lie}}(\gamma)$, 
which proves injectivity.

Lemma~\ref{invariant-coefficient-complex} and
Lemma~\ref{stabilizer-quasi-isomorphism} now give
\[
\begin{aligned}
H^\bullet(M,\mathcal F)
&\cong
H^\bullet\bigl(
\Omega^\bullet(M,\mathcal F)^{\mathfrak g}
\bigr)\\
&\cong
H^\bullet\bigl(
(\wedge^\bullet\mathfrak g^*)^{\widehat K},
d_{CE}
\bigr)\\
&\cong
H^\bullet(\mathfrak g).
\end{aligned}
\]
\end{proof}

Historically, the conclusion of Theorem~\ref{thm:liefoliation} appears to be
well known in the case where a complete Lie foliation $\mathcal F$ on a manifold $M$ has a dense leaf. Let $G$ be the
simply connected Lie group with Lie algebra $\mathfrak g$. By
Fedida's Darboux covering construction for a complete Lie $\mathfrak g$-foliation
(see \cite[Chapter~4.2]{Mo88}), the basic de Rham complex
$(\Omega^\bullet(M,\mathcal F),d)$ is naturally identified with $(\Omega^\bullet(G)^\Gamma,d)$,
where $\Gamma$ is a discrete subgroup of $G$---called the holonomy group of the
Darboux covering---acting on $G$ on the left. When $\mathcal F$ has a dense
leaf, $\Gamma$ is dense in $G$, and a continuity argument yields
\[
(\Omega^\bullet(M,\mathcal F),d)
\;\cong\;(\Omega^\bullet(G)^\Gamma,d)
\;=\;(\Omega^\bullet(G)^G,d)
\;=\;(C_{CE}^\bullet(\mathfrak g),d),
\]
so the basic cohomology of $\mathcal F$ coincides with $H^\bullet(\mathfrak g)$.

A second case, perhaps less well known and apparently not recorded in the
literature, arises when $\mathfrak g$ is semisimple of compact type. Then
$\mathfrak g$ integrates to a compact simply connected Lie group $G$, and
$H:=\overline{\Gamma}$ is a compact and connected subgroup of $G$. Since
$\Omega^\bullet(G)^\Gamma=\Omega^\bullet(G)^H$, averaging over the compact group
$H$ shows that
\[
H^\bullet(\Omega^\bullet(G)^H,d)\cong H^\bullet(\mathfrak g).
\]
The main novelty of Theorem~\ref{thm:liefoliation} lies in the general compact
type case, where $\mathfrak g$ may possess a nontrivial center.

\begin{ex}
We show that the compact-type hypothesis in
Theorem~\ref{thm:liefoliation} is essential, even for the point foliation.

Let
\[
L=\mathrm{PSL}_2(\mathbb R),
\]
and let $\Gamma\subset L$ be a torsion-free cocompact lattice. Set
\[
M=\Gamma\backslash L,
\]
and let $\mathcal F$ be the point foliation on $M$. Then $(M,\mathcal F)$ is a Lie $\mathfrak l$-foliation, where
$
\mathfrak l=\operatorname{Lie}(L)=\mathfrak{sl}_2(\mathbb R)$. 

Indeed, the right Maurer--Cartan form on $L$ is left-invariant, hence descends
to an $\mathfrak l$-valued $1$-form on $M$. Since $\mathcal F$ is the point
foliation, this descended form has kernel $0=T\mathcal F$, is pointwise an
isomorphism, and satisfies the Maurer--Cartan equation.

The leaf closure space of $(M,\mathcal F)$ is just $M$ itself, and is therefore
compact. On the other hand, $\mathfrak l=\mathfrak{sl}_2(\mathbb R)$ is not of
compact type.

Since $\mathcal F$ is the point foliation, its basic cohomology coincides with
ordinary de Rham cohomology:
\[
H^\bullet(M,\mathcal F)=H^\bullet_{\mathrm{dR}}(M).
\]
Moreover, $M$ is diffeomorphic to the unit tangent bundle $T^1\Sigma$ of a
closed hyperbolic surface $\Sigma$. Hence, by the Gysin sequence of the circle
bundle
\[
S^1\longrightarrow T^1\Sigma\longrightarrow \Sigma,
\]
we have
\[
H^1(M,\mathcal F)=H^1_{\mathrm{dR}}(M)\neq 0.
\]
By contrast, Whitehead's lemma gives
\[
H^1(\mathfrak l)=H^1(\mathfrak{sl}_2(\mathbb R))=0.
\]
Therefore
\[
H^\bullet(M,\mathcal F)\not\cong H^\bullet(\mathfrak l).
\]

Thus the conclusion of Theorem~\ref{thm:liefoliation} fails without the
compact-type hypothesis, even in the case of the point foliation.
\end{ex}

\section{The de Rham cohomology of $G/H$}
\label{De Rham of G/H}

Throughout this section, we use $L$ and $R$ to denote left and right group
actions, respectively. We write $\mathfrak g$ and $\mathfrak h$ for the Lie
algebras of a Lie group $G$ and a Lie subgroup $H$, respectively, and let
$\mathcal F$ denote the foliation on $G$ whose leaves are the orbits of the
right $H$-action. The main results of this section are
Theorems~\ref{thm:GH1} and~\ref{thm:GH}.

We write $H^\bullet_{\mathrm{dR}}(X)$ for the diffeological de Rham cohomology
of a diffeological space $X$. When $X$ is a smooth manifold equipped with its
canonical manifold diffeology, the diffeological de Rham complex is canonically
isomorphic to the usual de Rham complex; thus this notation agrees with ordinary
de Rham cohomology in the manifold case. We refer the reader to \cite{L26} for
the minimal background on diffeology needed here, and to \cite{I13} for a
detailed exposition.

A theorem of Hector, Mac\'{\i}as-Virg\'os, and
Sanmart\'{\i}n-Carb\'on~\cite[Theorem~3.5]{HMS} identifies the diffeological
de Rham cohomology of the leaf space of a foliation, equipped with its quotient
diffeology, with the basic cohomology of the foliation. In the present setting,
the leaf space is precisely the diffeological quotient $G/H$. The computation
of $H^\bullet_{\mathrm{dR}}(G/H)$ therefore reduces to a computation in basic
cohomology, for which the transverse averaging mechanism developed above is
the key new ingredient. We begin with a technical lemma.

\begin{lem}\label{basic-relative}
Let $H$ be a connected Lie subgroup of a Lie group $G$. Then restriction to
the identity induces a canonical isomorphism of cochain complexes
\[
\Omega^\bullet(G,\mathcal F)^{L(G)} \cong C^\bullet_{CE}(\mathfrak g,\mathfrak h).
\]\end{lem}

\begin{proof}
Restriction to the identity induces the standard isomorphism of cochain complexes
\[
\Omega^\bullet(G)^{L(G)} \xrightarrow{\;\cong\;} \wedge^\bullet\mathfrak g^*,
\qquad \omega \mapsto \omega_e,
\]
under which the de Rham differential corresponds to the Chevalley--Eilenberg differential.

Since the leaves of $\mathcal{F}$ are the right $H$-orbits, we have
\[
T_g\mathcal F=(dL_g)_e(\mathfrak h).
\]
Equivalently, for each $X\in\mathfrak h$, the fundamental vector field on $G$
generated by the right $H$-action is
\[
\widetilde X(g):=\left.\frac{d}{dt}\right|_{t=0} g\exp(tX)=(dL_g)_eX.
\]
Thus $\mathcal{F}$ is generated by the vector fields $\widetilde X$, with $X\in\mathfrak h$.
Hence a left-invariant form $\omega$ is $\mathcal{F}$-basic if and only if
\[
\iota_{\widetilde X}\omega=0,\qquad \mathcal L_{\widetilde X}\omega=0
\quad\text{for all }X\in\mathfrak h.
\]
Under restriction to the identity, these conditions become
\[
\iota_X(\omega_e)=0,\qquad \operatorname{ad}^*(X)(\omega_e)=0
\quad\text{for all }X\in\mathfrak h.
\]
Therefore $\Omega^\bullet(G,\mathcal F)^{L(G)}$ identifies with
$C^\bullet_{CE}(\mathfrak g,\mathfrak h)$.
\end{proof}

We now proceed to the proof of Theorem~\ref{thm:GH1}.

\begin{proof}[Proof of Theorem~\ref{thm:GH1}]
Since $\mathfrak g$ is of compact type, it admits an
$\operatorname{ad}$-invariant positive-definite inner product. As $G$ is
connected, this inner product is $\operatorname{Ad}(G)$-invariant and therefore
defines a bi-invariant Riemannian metric $g_G$ on $G$. In particular, $g_G$ is
complete.

The leaves of $\mathcal F$ are the orbits of the right $H$-action. Since right
translations by elements of $H$ are isometries of $g_G$, the metric $g_G$ is
bundle-like with respect to $\mathcal F$. Thus $(G,\mathcal F)$ is a complete
Riemannian foliation. The left action of $G$ on itself commutes with the right $H$-action and
therefore preserves $\mathcal F$. Since left translations are also isometries
of $g_G$, they preserve the induced transverse Riemannian metric.
Consequently, differentiating the left $G$-action yields an isometric
transverse $\mathfrak g$-action on $(G,\mathcal F)$.

The closure of the leaf through $g\in G$ is $g\overline H$, so the leaf
closure space is $G/\overline H$, which is compact by assumption. Hence all
the hypotheses of Theorem~\ref{average-operator} are satisfied. The left action of $G$ preserves the foliation and hence acts on
$\Omega^\bullet(G,\mathcal F)$. Since $G$ is connected, a basic form is
invariant under the transverse $\mathfrak g$-action described above if and only if it is invariant under all left
translations. Thus
\[
\Omega^\bullet(G,\mathcal F)^{\mathfrak g}
=
\Omega^\bullet(G,\mathcal F)^{L(G)},
\]
and the transverse averaging operator takes values in this common
subcomplex.
It follows from Theorem~\ref{average-operator} that the inclusion
\[
\bigl(\Omega^\bullet(G,\mathcal F)^{L(G)},d\bigr)
\hookrightarrow
\bigl(\Omega^\bullet(G,\mathcal F),d\bigr)
\]
is a quasi-isomorphism.

By Lemma~\ref{basic-relative}, restriction to the identity gives an
isomorphism of cochain complexes
\[
\bigl(\Omega^\bullet(G,\mathcal F)^{L(G)},d\bigr)
\cong
\bigl(C^\bullet_{CE}(\mathfrak g,\mathfrak h),d_{CE}\bigr).
\]
On the other hand, the theorem of Hector, Mac\'{\i}as-Virg\'os, and
Sanmart\'{\i}n-Carb\'on~\cite[Theorem~3.5]{HMS} gives
\[
H^\bullet_{\mathrm{dR}}(G/H)
\cong
H^\bullet(G,\mathcal F).
\]
Combining these identifications, we obtain
\[ H^\bullet_{\mathrm{dR}}(G/H)
\cong H^\bullet(G,\mathcal F) 
\cong
H^\bullet\bigl(\Omega^\bullet(G,\mathcal F)^{L(G)},d\bigr) 
\cong H^\bullet(\mathfrak g,\mathfrak h).
\]
\end{proof}

The following result shows that when $\mathfrak h$ is an ideal of
$\mathfrak g$, the compactness assumption on $G/\overline{H}$ in
Theorem~\ref{thm:GH} automatically implies that $\mathfrak g/\mathfrak h$ is of
compact type.

\begin{lem}\label{relationship}
Let $G$ be a connected Lie group and $H$ a connected Lie subgroup. Assume that
$\mathfrak h$ is an ideal of $\mathfrak g$ and that $G/\overline{H}$ is compact.
Then $\mathfrak g/\mathfrak h$ is of compact type.
\end{lem}

\begin{proof}
Since $G$ is connected and $\mathfrak h$ is an ideal of $\mathfrak g$, the
adjoint action of $G$ preserves $\mathfrak h$, and hence induces a
representation
\[
\overline{\mathrm{Ad}}:G\to \mathrm{GL}(\mathfrak g/\mathfrak h),\qquad
\overline{\mathrm{Ad}}_g(X+\mathfrak h)=\mathrm{Ad}_g(X)+\mathfrak h.
\]
If $\xi\in \mathfrak h$, then
\[
d(\overline{\mathrm{Ad}})_e(\xi)(X+\mathfrak h)=[\xi,X]+\mathfrak h=0,
\]
because $\mathfrak h$ is an ideal. Thus $\overline{\mathrm{Ad}}$ is trivial on the
connected subgroup $H$, and therefore also on $\overline{H}$ by continuity.
Hence $\overline{\mathrm{Ad}}$ factors through the compact quotient
$G/\overline{H}$. It follows that the image
$K:=\overline{\mathrm{Ad}}(G)$ is a compact subgroup of
$\mathrm{GL}(\mathfrak g/\mathfrak h)$.

Choose any inner product $\langle \cdot,\cdot\rangle_0$ on
$\mathfrak g/\mathfrak h$, and average it over $K$:
\[
\langle u,v\rangle:=\int_K \langle ku,kv\rangle_0\,dk.
\]
Then $\langle \cdot,\cdot\rangle$ is a positive definite $K$-invariant inner
product on $\mathfrak g/\mathfrak h$. Differentiating the $K$-invariance shows
that it is $\mathrm{ad}$-invariant. Therefore $\mathfrak g/\mathfrak h$ is of
compact type.
\end{proof}

\begin{lem}\label{isometry}
Suppose that $G$ is connected, $H$ is a connected Lie subgroup of $G$, $\mathfrak h$ is an ideal of $\mathfrak g$, and $G/\overline{H}$ is compact. Then $(G,\mathcal F)$ is a complete Lie $(\mathfrak g/\mathfrak h)$-foliation of compact type.
\end{lem}

\begin{proof}
Let $\theta_G$ be the right Maurer--Cartan form on $G$, and let
$p:\mathfrak g\to \mathfrak g/\mathfrak h$ be the natural quotient map.
Define a $(\mathfrak g/\mathfrak h)$-valued $1$-form $\alpha$ on $G$ by
\[
\alpha = p\circ \theta_G.
\]
It is straightforward to verify that $\alpha$ satisfies the Maurer--Cartan equation in the sense of Definition~\ref{lie-g-foliation}. Hence $\alpha$ endows $(G,\mathcal F)$ with the structure of a complete Lie $(\mathfrak g/\mathfrak h)$-foliation. The fact that this foliation is of compact type follows from Lemma~\ref{relationship}.
\end{proof}

We are now ready to prove Theorem~\ref{thm:GH}.

\begin{proof}[Proof of Theorem~\ref{thm:GH}]
Note that the leaf space of the foliation $\mathcal F$ is precisely the
homogeneous space $G/H$. It follows from \cite[Theorem 3.5]{HMS} that the
diffeological de Rham cohomology $H^\bullet_{\mathrm{dR}}(G/H)$ is canonically
isomorphic to the basic cohomology $H^\bullet(G,\mathcal F)$.

By Lemma~\ref{isometry}, $(G,\mathcal F)$ is a complete Lie
$(\mathfrak g/\mathfrak h)$-foliation with compact leaf closure space.
Therefore Theorem~\ref{thm:liefoliation} gives
\[
H^\bullet(G,\mathcal F)\cong H^\bullet(\mathfrak g/\mathfrak h).
\]
Combining these two identifications yields
\[
H^\bullet_{\mathrm{dR}}(G/H)\cong H^\bullet(\mathfrak g/\mathfrak h).
\]
\end{proof}

The following two examples serve complementary purposes. The first is a
counterexample showing that the compactness hypothesis in
Theorem~\ref{thm:GH} is sharp: without this condition, the conclusion may
fail even in simple cases. The second is a concrete positive example
showing that the theorem applies to connected Lie subgroups that are
neither closed nor dense, even when $\mathfrak g/\mathfrak h$ is
nonabelian and $G$ is noncompact. Throughout, we write
$H^\bullet_{\mathrm{dR}}(G/H)$ for the diffeological de Rham cohomology
of $G/H$ equipped with its quotient diffeology.
\begin{ex}
We present an example showing that compactness of the
quotient $G/\overline H$ is essential. Let
$G=\mathbb R\times SU(2)$ and $H=SU(2)$. Then $G$ is connected,
$H$ is a connected closed normal Lie subgroup, and
$\mathfrak g=\mathbb R\oplus \mathfrak{su}(2)$,
$\mathfrak h=\mathfrak{su}(2)$, so
$\mathfrak g/\mathfrak h\cong \mathbb R$, which is of compact type.
However, $\overline H=H$ and $G/H\cong \mathbb R$, so
$G/\overline H$ is noncompact.

Since $G/H\cong \mathbb R$, we have
$H^1_{\mathrm{dR}}(G/H)=0$, whereas
$H^1(\mathfrak g/\mathfrak h)\cong \mathbb R$. Thus
$H^\bullet_{\mathrm{dR}}(G/H)\not\cong H^\bullet(\mathfrak g/\mathfrak h)$.
\end{ex}

\begin{ex}
Let $V=\mathbb C^3$, regarded as a real vector space, and let
$K:=U(3)\times S^1$ act on $V$ by $(B,z)\cdot v=zBv$. Set
$G:=V\rtimes K$. Then $G$ is connected and noncompact.

Let $I_3$ denote the $3\times3$ identity matrix. The center of $K$ is
$T:=Z(K)=\{(\zeta I_3,\eta):\zeta,\eta\in S^1\}$. Choose
$\lambda\in\mathbb R\setminus\mathbb Q$, and set
\[
L:=\{(e^{2\pi it}I_3,e^{2\pi i\lambda t}):t\in\mathbb R\}\subset T,
\qquad H:=V\rtimes L\subset G.
\]
Since $\lambda$ is irrational, $L$ is dense in $T$. Thus $H$ is a
connected immersed Lie subgroup of $G$ and
$\overline H=V\rtimes T$. In particular, $H$ is neither closed nor dense
in $G$, and $G/H$ is non-Hausdorff.

We verify the hypotheses of Theorem~1.5. Write
$\mathfrak g=V_{\mathbb R}\rtimes(\mathfrak u(3)\oplus\mathbb RE)$,
where $E$ generates the Lie algebra of the second $S^1$-factor. Let
$Z=iI_3$, so that
$\mathfrak u(3)=\mathfrak{su}(3)\oplus\mathbb RZ$, and set
$\Xi:=Z+\lambda E$. The Lie algebra of $H$ is
$\mathfrak h=V_{\mathbb R}\rtimes\mathbb R\Xi$.

Since $\Xi$ is central in $\mathfrak u(3)\oplus\mathbb RE$, for
$X\in\mathfrak u(3)\oplus\mathbb RE$, $v,w\in V_{\mathbb R}$, and
$a\in\mathbb R$, we have
\[
[(v,X),(w,a\Xi)]
=(X\cdot w-a\Xi\cdot v,0)\in V_{\mathbb R}\subset\mathfrak h.
\]
Hence $\mathfrak h$ is an ideal in $\mathfrak g$.

The homomorphism $\pi:G\to PU(3)$ given by $\pi(v,B,z)=[B]$ has kernel
$V\rtimes T=\overline H$. Therefore $G/\overline H\cong PU(3)$, which
is compact. Furthermore,
\[
\mathfrak g/\mathfrak h
\cong(\mathfrak u(3)\oplus\mathbb RE)/\mathbb R\Xi
\cong\mathfrak{su}(3)\oplus\mathbb R.
\]
Theorem~1.5 therefore applies. Since $\mathfrak h$ is an ideal in
$\mathfrak g$, the relative Chevalley--Eilenberg complex of
$(\mathfrak g,\mathfrak h)$ identifies naturally with the
Chevalley--Eilenberg complex of $\mathfrak g/\mathfrak h$. Consequently,
\[
H^\bullet_{\mathrm{dR}}(G/H)
\cong H^\bullet(\mathfrak g/\mathfrak h)
\cong H^\bullet(\mathfrak g,\mathfrak h)
\cong\Lambda(\alpha_1,\alpha_3,\alpha_5),
\qquad \deg\alpha_j=j,
\]
where $\alpha_1$ is the degree-one generator arising from the
$\mathbb R$-summand of $\mathfrak g/\mathfrak h\cong
\mathfrak{su}(3)\oplus\mathbb R$, while $\alpha_3$ and $\alpha_5$ are
the standard primitive generators of $H^\bullet(\mathfrak{su}(3))$.\end{ex}

\end{document}